\documentclass{amsart}
\usepackage{amsmath,amsthm,amsfonts,amssymb,amscd,euscript,enumerate}
\usepackage{mathtools}
\usepackage{cite}
\usepackage{pgfplots}
\usetikzlibrary{patterns,arrows,positioning,calc,fadings,shapes,decorations.markings}
\usepackage{tikz-cd}
\usepackage{graphicx}
\usepackage{quiver}
\theoremstyle{plain}
\newtheorem{thm}{Theorem}[section]
\newtheorem{theorem}[thm]{Theorem}
\newtheorem{lemma}[thm]{Lemma}
\newtheorem{corollary}[thm]{Corollary}
\newtheorem{proposition}[thm]{Proposition}

\theoremstyle{definition}

\newtheorem{remark}[thm]{Remark}

\newtheorem{definition}[thm]{Definition}

\newtheorem{defn-thm}[thm]{Definition-Theorem}

\begin{document}
	\title{quantum cluster variables via canonical submodules}
	\author{Fan xu, Yutong Yu*}
	\address{Department of Mathematical Sciences\\
		Tsinghua University\\
		Beijing 100084, P.~R.~China}
	\email{fanxu@mail.tsinghua.edu.cn(F. Xu)}
	\address{Department of Mathematical Sciences\\
		Tsinghua University\\
		Beijing 100084, P.~R.~China}
	\email{yyt20@mails.tsinghua.edu.cn(Y. Yu)}
	\keywords{Quantum cluster algebra, Triangulated category, Gentle algebra, Cluster character}
	\thanks{$*$~Corresponding author.}
	\begin{abstract}
		We study quantum cluster algebras from marked surfaces without punctures. We express the quantum cluster variables in terms of the canonical submodules. As a byproduct, we obtain the positivity for this class of quantum cluster algebra.
	\end{abstract}
	
	\maketitle
\section{INTRODUCTION}
\par Cluster algebras were introduced by Fomin and Zelevinsky around 2000\cite{Fomin_2001}. Later, Berenstein and Zelevinsky introduced quantum cluster algebras\cite{BZ04}. The theory of cluster algebra is related to many other branches of mathematics, such as Lie theory, representation theory of algebras, Teichm{\"u}ller theory and mathematical physics.
\par The cluster algebras associated with marked surfaces were studied in \cite{Fomin2006ClusterAA}. A key part of the main result is that there is a bijection between the cluster variables and arcs in the surface which remains true in quantum case. Later, \cite{Musiker_2011} provided a Laurent expansion formula for arcs with respect to any triangulation using combinatorial tools called snake graphs and perfect matchings. After that, \cite{Huang2021AnEF} gave a quantum version of the expansion formulas and proved the positive conjecture in \cite{Fomin_2001} for cluster algebra associated to unpunctured surface. The expansion formulas are sums over the perfect matchings of snake graphs which are determined by the crossings of arcs and triangulations.
\par  The Jacobian algebras associated with the quivers with potentials given by the triangulations of surfaces are gentle algebras. The module categories of gentle algebras are studied in \cite{Butler1987AuslanderreitenSW} and the cluster categories $C$ of quiver with potential are studied in \cite{Amiot}. A correspondence between (closed) curves on surfaces and objects in cluster categories was given by \cite{Brustle2010OnTC}. Specifically, the string objects in $C$ are indexed by curves and can be identified with string modules in the gentle algebra associate with a triangulation. Palu \cite{Palu2008} defined a cluster character for any object in some 2-Calabi-Yau category. The cluster character is a generalization of $CC$-map which was proposed in\cite{Caldero2004ClusterAA}. Later Qin \cite{Qin_2012} gave a quantum Laurent polynomial of any cluster variables using Serre polynomials. All the characters are sums over the dimension vector and depend on the geometric properties of quiver Grassmannians including their Euler characteristics.
\par Then the connection between expansion formulas of  arcs and cluster characters of modules should be examined as both provide Laurent expansion formulas of cluster variables in the same cluster algebra. The answer to this is found in a comparison between perfect matchings and canonical submodules in \cite{Brstle2011AMI}. Using the results in \cite{Haupt2012EulerCO}, we know the Euler characteristics of quiver Grassmannians referring to string modules are equal to the numbers of canonical submodules of string modules. Therefore, the cluster character can be expressed as a summation over canonical submodules.
\par The aim of this paper is to provide a quantum Laurent expansion formula for quantum cluster variables in quantum cluster algebras from unpunctured surfaces using the canonical submodules. The key point in our expansion formula is the weight functions of canonical submodules, which take values on $\mathbb{Z}$. Our strategy of proving the main theorem is similar to Min Huang's approach in\cite{Huang2021AnEF}. However, we are not yet able to establish the relation between the weight functions here and the valuation maps in Huang' paper. We state the main theorem below.
\par For any string objects $M$ and any cluster-tilting objects related by mutation $T'=\mu_k(T)$. Let the string module of $M$ in $End(T)$(or $End(T')$) be $M(w)$(or $M(w')$) where $w$(or $w'$) is the string. The set of canonical submodules of $M(w)$(or $M(w')$) is denoted by $CS(M(w))$(or $CS(M(w'))$).Then, we have the following
\begin{theorem}
	There exist functions $v$ and $v'$ valued on $\mathbb{Z}$ satisfying the following polynomial property:
	\begin{equation*}
		\sum_{U\in CS(M(w))}q^{\frac{v(U)}{2}X^{ind_T(r)+B_T dim(U)}}=\sum_{U'\in CS(M(w'))}q^{\frac{v'(U')}{2}(X')^{ind_{T'}r+B_{T'} dim(U')}}.
	\end{equation*}
\end{theorem}
The details of $v$ and $v'$ will be explained in Section 5. Roughly speaking, we split the string $w$(or $w'$) into several subwords such that in each subword, there is an easy way to define the function on both sides. The function on the whole string is then a sum of functions on the subwords, plus some additional terms. See Definition 5.23 and the case-by-case analysis in Section 5 for details.
\par The following theorem is a corollary of the above
\begin{theorem}
	There exists a function $v$ valued on $\mathbb{Z}$ such that the following expansion formula for the quantum cluster variable corresponding to $M$ with respect to $T$ holds:
	\begin{equation*}
		X_M^T=\sum_{N\in CS(M)}q^{\frac{v(N)}{2}}X^{ind_T(M)+B_TN}.
	\end{equation*}
\end{theorem}
Since the equation here is manifestly positive, we obtain the positivity property for quantum cluster algebras from unpunctured surface. 

\section{2-CALABI-YAU TRIANGULATED CATEGORY}
Let $\mathcal{C}$ be a Hom-finite,2-Calabi-Yau,Krull-Schmidt triangulated category with a cluster tilting object \textit{T} over a algebraically closed  field. Denote by $[1]$ its shift functor. Let \textit{B} be the endomorphism algebra of \textit{T} and $Q$ be the Gabriel quiver of $B$, then there exists a functor
\begin{equation}
	F'=Hom(T,-): \mathcal{C} \rightarrow modB \nonumber.
\end{equation}
\par Let $T=T_1\oplus T_2\oplus ...\oplus T_n$ where $T_1,...,T_n$ are pairwise non-isomorphic indecomposable direct summands of \textit{T} and $add(T)$ be the full subcategory of $\mathcal{C}$ where objects are direct sums of summands of $T$. Let $S_i$ be the top of $P_i=F'T_i$, for $i=1,2,...,n$ and $proj(B)$ be the subcategory of projective modules in $mod(B)$. The  Grothendieck group of \textit{modB} is denoted by \textit{K}$_0(modB)$. Denote the equivalence class in \textit{K}$_0(modB)$  of an object $X$ by $[X]$.
\par For any objects $M$ and $N$ in $modB$, define
\begin{equation}
	<M,N>=dim_kHom_B(M,N)-dim_kExt_B^1(M,N)\nonumber,
\end{equation}
\begin{equation}
	<M,N>_a=<M,N>-<N,M>\nonumber.
\end{equation}
and \cite{Palu2008} shows $<-,->_a$ can descend to the Grothendieck group.
\par For any object $X$ in $\mathcal{C}$, we have the following triangle
\begin{equation}
	T_1^X\rightarrow T_0^X\rightarrow X\rightarrow T_1^X[1] \nonumber 
\end{equation}
where $T_1^X,T_0^X\in add(T)$. So we can define the index of $X$
\begin{equation}
	ind_TX=[F'T_0^X]-[F'T_1^X] \nonumber.
\end{equation}
\par Sometimes, we omit the subscript when there is no ambiguity. The functor $F'$ induces an equivalence between $add(T)$ and $proj(B)$, as well as between $K_0(add(T))$ and $K_0(proj(B))$. Under this equivalence, we always identify them. Each cluster tilting object $T$ determines a cluster tilting subcategory $add(T)$, and we often use the cluster tilting object to represent its full additive subcategory. 
\begin{definition}
	The cluster tilting objects define a cluster structure on $\mathcal{C}$ if the following holds
	\begin{enumerate}
		\item For each cluster tilting object $T$ and any indecomposable direct summand $M$ where $T=M\oplus L$, there exists a unique(up to isomorphism) indecomposable object $M'$ such that $T'=M'\oplus L$ is a cluster tilting object and we denote $T'=\mu_M(T)$.
		\item There are non-split triangles
		\begin{equation*}
			M'\stackrel{f}{\rightarrow}E\stackrel{g}{\rightarrow}M\rightarrow M'[1] \ and \ M\stackrel{s}{\rightarrow}E'\stackrel{t}{\rightarrow}M'\rightarrow M[1]
		\end{equation*}
		where $g$ and $t$ are minimal right $T\cap T'$-approximates and $f$ and $s$ are minimal left $T\cap T'$-approximates.
		\item For any cluster tilting object $T$, the quiver $Q(T)$ of the endomorphism algebra of $T$ has no loops or 2-cycles.
		\item If $T'=\mu _M(T)$, then $Q(T')=\mu_M(Q(T))$ where the $\mu_M$ here is the Fomin-Zelevinsky mutation in direction $M$.
	\end{enumerate}
\end{definition}
Now assume that $\mathcal{C}$ admits a cluster structure. Then we have a $n$-regular tree $\mathbb{T}_n$. Let $t_0$ be a fixed vertex of $\mathbb{T}_n$. For each vertex $t$ of $\mathbb{T}_n$, we associate a cluster tilting object $T_t$ subject to the following rule:
\begin{enumerate}
	\item $T_{t_0}=T$,
	\item if $t$ is adjacent to $t'$ by an edge labeled $k$, the $T_{t'}=\mu_k(T)$.
\end{enumerate}
Then we call a cluster tilting object $T'$ reachable if it can be obtained by a finite sequence of mutation of $T$, and an indecomposable object $M$ is reachable if it is a direct summand of a reachable cluster tilting object $T'$.
\par A cluster character on $\mathcal{C}$ is a map from the objects of $\mathcal{C}$ to a commutative ring $A$, usually we take $A=Q(x_1,...,x_n)$. For any object $M$ in $\mathcal{C}$, \cite{Palu2008} defines the cluster character as the following:
\begin{equation}
	x^T_M=\sum_{e}\chi(Gr_eF'M)X^{ind_T(M)+B_Te} \nonumber.
\end{equation}

\section{QUANTUM CLUSTER ALGEBRA}
In this section, we recall some definitions and properties of quantum cluster algebra. Let $q$ be a formal variable, $B$ be an $m\times n$ matrix with full rank and $\Lambda$ be a skew-symmetric matrix satisfying
\begin{equation}
	\Lambda(-B)=\begin{bmatrix}
		D \\0
	\end{bmatrix} \nonumber
\end{equation}
where $D$ is a diagonal matrix with positive entries. We call $(B,\Lambda)$ a compatible pair. Since we focus on cluster algebras from surface, we can assume the matrix $D$ is the identity matrix.
\begin{definition}
	Let $L$ be a lattice of rank $m$ with a basis $\left\{e_i|1\leq i\leq m \right\}$ with a skew-symmetric bilinear form $\Lambda$. The quantum torus $\mathcal{T}=\mathcal{T}(L,\Lambda)$ is the $\mathbb{Z}$-algebra generated by $X^g$, $g\in L$ subject to the relation
	\begin{equation*}
		X^gX^h=q^{\Lambda(g,h)/2}X^{g+h}.
	\end{equation*}
\end{definition}
\par An initial quantum seed is a triple $(\Lambda,B,X)$ such that the pair $(B,\Lambda)$ is compatible and $X=\left\{X_1,...,X_m\right\}$ where $X_i=X^{e_i}$ for $1\leq i\leq m$.
\par For any $1\leq k\leq n$, we can define the mutation of the quantum seed in direction $k$. Denoted by $\mu _k(\Lambda,B,X)=(\Lambda^{\prime},B^{\prime},X^{\prime})$, where
\begin{enumerate}
	\item \begin{equation*}
		\Lambda'_{ij}=\begin{cases}
			\Lambda_{ij},      &\mbox{if $i,j\neq k$}, \\
			\Lambda(e_i,-e_k+\sum_{l}{[b_{lk}]_+}e_l), &j=k\neq i.\\
		\end{cases}
	\end{equation*}
	\item \begin{equation*}
		b^{\prime}_{ij}=\begin{cases}
			-b_{ij} &{\rm if} \ i=k \ {\rm or} \ j=k,\\
			b_{ij}+[b_{ik}]_+b_{kj}+b_{ik}[-b_{kj}]_+ &otherwise.
		\end{cases}
	\end{equation*}
	where $B=(b_{ij})$ and $B^{\prime}=(b^{\prime}_{ij})$.
	\item $X'=\left\{X'_1,...,X'_m\right\}$ is given by
	\begin{equation*}
		X'_k=X^{-e_k+\sum_{i}{\left[b_{ik}\right]_+e_i}}+X^{-e_k+\sum_{i}{\left[-b_{ik}\right]_+e_i}}.
	\end{equation*}
	where $X'_i=X_i$ for $i\neq k$.
\end{enumerate}
It is proved in \cite{BZ04} that the mutation here can be applied on the new seed on any direction. If we collect all the seed which can obtained from a finite sequence of mutations on the initial seed, we can build a $n$-regular tree $\mathbb{T}_n$ where the vertices of $\mathbb{T}_n$ match the quantum seed and edges match mutation in some direction. We denote $t_0$ the vertex of the initial seed. For any vertex $t$ in $\mathbb{T}_n$, the quantum seed $X(t)=(X_1(t),\cdots,X_m(t))$ can be obtained by a finite sequence of mutation in direction $k_1,\cdots,k_s$ and we denote by $t=\mu_{k_s}\cdots \mu_{k_1}(t_0)$.
\begin{definition}
	The quantum cluster algebra $\mathcal{A}_q$ is a subalgebra of $\mathcal{T}$ generated by $X_i(t)$ for any $1\leq i\leq m$ and $t\in \mathbb{T}_n$.
\end{definition}
By definition, all $X_i(t)$ equal for any $n+1\leq i\leq m$ and $t\in \mathbb{T}_n$ and we can denote $X_{n+1},\cdots,X_m$ be the common value.
\begin{enumerate}
	\item $X_i(t)$ is called quantum cluster variable for $1\leq i\leq n$ and $t\in \mathbb{T}_n$.
	\item $X_i$ is called frozen variable for $n+1\leq i\leq m$.
	\item $X(t)$ is called a cluster.
	\item For any $t\in \mathbb{T}_n$, $X(t)^a$ is called a quantum cluster monomial for some $a\in \mathbb{N}^m$.
\end{enumerate}
\begin{theorem}(Quantum Laurent Phenomenon)
	Any quantum cluster variable is a Laurent polynomial in $X_i(t), 1\leq i\leq m$ for some fixed $t\in \mathbb{T}_n$.
\end{theorem}

\section{CLUSTER CATEGORY OF SURFACE}
\subsection{Gentle algebra}A finite dimensional algebra $A$ is gentle if $A=kQ/I$ satisfying the following conditions:
\begin{enumerate}
	\item There are at most 2 arrows start or end at each vertices.
	\item $I$ is generated by paths of length 2.
	\item For each arrow $\alpha$, there is at most one arrow $\beta$ and at most one arrow $\gamma$ such that $\alpha \beta\in I$ and $\gamma \alpha \in I$.
	\item For each arrow $\alpha$, there is at most one arrow $\beta$ and at most one arrow $\gamma$ such that $\alpha \beta\notin I$ and $\gamma \alpha \notin I$.
\end{enumerate}
One of the crucial property of gentle algebra is the classification of indecomposable nodules over $A$. It is given by the string and band in $A$ in the following way:
\par A $string$ is a walk $w$ in $Q$ 
\begin{equation*}
	w=x_1\stackrel{\alpha_1}{\leftrightarrow} x_2\stackrel{\alpha_2}{\leftrightarrow}\cdots \stackrel{\alpha_{n-1}}{\leftrightarrow}x_n,
\end{equation*}
where the $x_i$ are the vertices of $Q$ and each $\alpha_i$ is an arrow between $x_i$ and $x_{i+1}$ in either direction such that $w$ does not contain a sequence of the form $\xleftarrow{\beta}\xrightarrow{\beta}$ or $\xrightarrow{\beta_1}\xrightarrow{\beta_2}\cdots \xrightarrow{\beta_s}$ with $\beta_s\cdots \beta_1\in I$ or their dual. A string is $cyclic$ if $x_1=x_n$. A $band$ is a cyclic string $b$ such that each power $b^n$ is a string, but $b$ is not a proper power of some string.
\par The string module $M(w)$ is obtained from the string $w$ by replacing each $x_i$ by a copy of the field $k$. The action of an arrow $\alpha$ on $M(w)$ is induced by the relevant identity morphism if $\alpha$ lies on $w$, and is zero otherwise. Since the string $w$ and $w^{-1}$ correspond to the same module, we always identify two inverse string.
\subsection{Marked surface} Most example of the gentle algebra is given by the marked surface. A marked surface is a pair $(S,M)$ where $S$ is an oriented surface and $M$ be a finite set of marked points intersecting each boundary component of $S$. The marked point in the interior of $S$ is called puncture. We only consider the unpunctured marked surface where all marked points lie on the boundary.
\begin{definition}
	An arc $\gamma$ in $(S,M)$ is an isotopy class of the curve in $S$ satisfying
	\begin{enumerate}
		\item the endpoints are in $M$;
		\item $\gamma$ does not have a self-crossing point except its endpoints;
		\item $\gamma$ does not intersect the boundary of $S$ or $M$ except its endpoints;
		\item $\gamma$ does not cut out an unpunctured monogon or bigon.
	\end{enumerate}
\end{definition}
\begin{definition}
	For any arcs $\gamma$ and $\gamma'$, denote $e(\gamma,\gamma')$ the minimal number of crossings between arcs $\alpha$ and $\alpha'$, where $\alpha$ and $\alpha'$ range over all arcs isotopy to $\gamma$ and $\gamma'$. We say $\gamma$ and $\gamma'$ are compatible if $e(\gamma,\gamma')=0$.
\end{definition}
\begin{definition}
	An (ideal) triangulation is a maximal collection of pairwise compatible arcs(together with all boundary segments).
\end{definition}
\noindent The arcs in $\Gamma$ cut $S$ into triangles and we call triangles whose edges are internal arcs internal triangles. We only consider marked surface $(S,M)$ admits a triangulation $\Gamma$. Then we can construct a quiver with potential $(Q(\Gamma),W(\Gamma))$ associate with $\Gamma$ in the following way:
\begin{enumerate}
	\item The vertices of the quiver $Q(\Gamma)$ are the internal arcs of $\Gamma$, and an arrow $i\rightarrow j$ exists if there is a triangle in $\Gamma$ containing two arcs $i$ and $j$ such that $i$ is clockwise to $j$.
	\item Each internal triangle $\Delta$ gives rise to a 3-cycle $\alpha_{\Delta}\beta_{\Delta}\gamma_{\Delta}$, then the potential is 
	\begin{equation*}
		W=\sum_{\Delta}\alpha_{\Delta}\beta_{\Delta}\gamma_{\Delta}.
	\end{equation*}
	where the sum range over all internal triangles.
\end{enumerate}
Let $A(\Gamma)$ be the Jacobian algebra of $(Q(\Gamma),W(\Gamma))$. Then we have the following lemma:
\begin{lemma}[\cite{Assem_2010}]
	Let $\Gamma$ be a triangulation of an unpunctured marked surface. Then $A(\Gamma)$ is a gentle algebra.
\end{lemma}
Since $A(\Gamma)$ is a gentle algebra, the indecomposable module over $A(\Gamma)$ is given by string and band in $Q(\Gamma)$. In addition, the string module is given by curves in $S$. For any curve $r$(considered as an isotopy class) which intersects the arcs of $\Gamma$ transversally such that the number of crossings is minimal. Following from one endpoint of $r$, denote the crossings orderly by $x_1,x_2\cdots x_s$, then we obtain a string $w(r): x_1\leftrightarrow x_2\leftrightarrow\cdots \leftrightarrow x_n$. In fact, this construction can obtain all string module.
\begin{lemma}[\cite{Brustle2010OnTC}]
	Let $\Gamma$ be a triangulation of $(S,M)$. Then there exists a bijection ${r}\mapsto w(r)$ between the homotopy classes of curves not homotopic to an arc in $\Gamma$ and the strings of $A(\Gamma)$.
\end{lemma}
Under this bijection, we can identify curves and strings emerging in our following section. We denote by $M(r)$ the string module corresponding to $w(r)$. Then the string with only one point labeled $i$ and no arrows matches the simple module $S_i$.
\begin{definition}
	A $flip$ turns a triangulation $\Gamma$ to another $\Gamma'$ such that $\Gamma'=\Gamma \backslash \left\{\gamma\right\}\cup \left\{\gamma'\right\}$ for some $\gamma\in \Gamma$ and denote by $\Gamma'=\mu_\gamma(\Gamma)$.
\end{definition} 
\begin{remark}
	For unpunctured surface $(S,M)$, flips can be done at any internal arcs of any triangulation. So if the triangulation $\Gamma$ has $n$ internal arcs, we can flip the triangulation at $n$ different direction and do the same on the new triangulation. So we can associate a $n$-regular tree $\mathbb{T}_n$ with the triangulation of $(S,M)$: the vertices of $\mathbb{T}_n$ are labeled by triangulations and edges are labeled by flips. Then if we take the $n$-regular tree  to be $\mathbb{T}_n$ associate with quantum cluster algebra and use the bijection in the above lemma, we can say the internal arcs not homotopy to arcs in $\Gamma$ correspond to quantum cluster variables which are not in the initial cluster. If we see the internal arcs as a empty string associate with some vertex, then it can be seen as an initial cluster variable. In this way, we identify a cluster variable with arcs(or its string) and a cluster with a triangulation. Then we can express the quantum cluster variables by string modules.
\end{remark}
\subsection{Cluster category of surface}For the quiver with potential $(Q(\Gamma),W(\Gamma))$, there is a corresponding Ginzburg dg-algebra $\Lambda_\Gamma$ and we denote by $D(\Lambda_\Gamma)$ its derived category.
\begin{definition}
	The $perfect \ derived \ category \ per(\Lambda_\Gamma)$ is the smallest thick subcategory of $D(\Lambda_\Gamma)$ containing $\Lambda_\Gamma$. The $finite \ dimensional \ derived \ category \ D^b(\Lambda_\Gamma)$ is the full subcategory of $D(\Lambda_\Gamma)$ of the dg-modules whose homology is of finite total dimension.
\end{definition} 
\noindent The $cluster \ category \ C_{Q(\Gamma),W(\Gamma)}$ is 
\begin{equation*}
	per(\Lambda_\Gamma)/D^b(\Lambda_\Gamma),
\end{equation*}
In addition, the cluster category does not depend on the triangulation of $(S,M)$, so we always denote $\mathcal{C}=C_{(S,M)}$ to be the cluster category associate with any triangulation.
\begin{lemma}
	$\mathcal{C}$ is Hom-finite, 2-Calabi-Yau ,Krull-Schmidt triangulated category. The image $T_\Gamma$ in $\mathcal{C}$ of the module $\Lambda$ is a cluster-tilting object. In addition, there is a equivalence
	\begin{equation*}
		F=Ext_{\mathcal{C}}^1(T_\Gamma,-):\mathcal{C}/add(T_\Gamma)\rightarrow modA(\Gamma).
	\end{equation*} 
\end{lemma}
From this lemma, we can say the indecomposable objects in $\mathcal{C}$ are either indecomposable module over $A(\Gamma)$ or direct summands of $T_\Gamma$. Actually, \cite{Brustle2010OnTC} proved 
\begin{lemma}
	A parametrization of the isoclasses of indecomposable object in $C$ is given by "string object" and "band object", where
	\begin{enumerate}
		\item the string objects are indexed by the homotopy class of curves in $(S,M)$ which are not homotopic to a boundary segment, subject to the equivalence relation $\gamma\sim \gamma^{-1}$
		\item the band objects are indexed by $k^*\times \Pi_1^*(S,M)/\sim$, where $k^*=k\backslash \left\{0\right\}$ and $\Pi_1^*(S,M)/\sim$ is given by the nonzero element of the fundamental group of $(S,M)$ subject to the equivalence relation generated by $a\sim a^{-1}$ and cyclic permutation.
	\end{enumerate}
\end{lemma}
\begin{remark}
	Let $r$ be a curve viewed as a string object in $\mathcal{C}$, then for any cluster-tilting object $T_\Gamma$, $Ext^1_{\mathcal{C}}(T_\Gamma,r)=M(w)$ where $M(w)$ is a string module over $A(\Gamma)$ corresponding to a string $w$ which is obtained from the intersections of $r$ and $\Gamma$. We sometimes say the index of $M(w)$ or $ind_{T_\Gamma}M(w)$ which actually means the index of $r$ with respect to $T_\Gamma$.
\end{remark}
The functor $Ext_{\mathcal{C}}^1(T_\Gamma,-)$ induces a canonical equivalence from $addT_\Gamma[1]$ to the injective module subcategory $inj(A(\Gamma))$. Under this equivalence, we can have the following lemma
\begin{lemma}[\cite{Brstle2011AMI}]
	Let $M(w)$ have a injective resolution
	\begin{equation*}
		0\longrightarrow M(w)\longrightarrow I_0\longrightarrow I_1
	\end{equation*}
	The index of $r$ is  equal to $[I_0]-[I_1]$
\end{lemma}
\subsection{Canonical submodule}In this section, we discuss the submodule of string modules.
\begin{definition}[\cite{anaki2018LatticeBF}]
	Given a string $w$ in $Q$ and its string module $M=M(w)$. We call $canonical \ embedding$ of a submodule $N$ of $M$, if the map $N\hookrightarrow M$ is induced by the identity map on the non-zero components of $N$. In this way, $N$ is called the $canonical \ submodule$ of $M$. We denote by $CS(M)$ the set of canonical submodule of $M$.
\end{definition}
\begin{definition}
	Let $w$ be a string in $Q$ 
	\begin{equation*}
		w=x_1\stackrel{\alpha_1}{\leftrightarrow} x_2\stackrel{\alpha_2}{\leftrightarrow}\cdots \stackrel{\alpha_{n-1}}{\leftrightarrow}x_n,
	\end{equation*}
	A subword $w_I$ of $w$ indexed by an interval $I=[j,s]$ is a string given by
	\begin{equation*}
		x_j\xleftrightarrow{\alpha_j} x_{j+1}\xleftrightarrow{\alpha_{j+1}} \cdots \xleftrightarrow{\alpha_{s-1}}x_s.
	\end{equation*}
	In case $M(w_I)$ is a submodule of $M(w)$, we call $w_I$ a substring of $w$.
\end{definition}
Moreover, for any subset $I\in \left\{1,2,\cdots,n\right\}$, we can write $I$ as a disjoint union of intervals of maximal length $I=I_1\cup I_2\cup\cdots \cup I_t$ that is
\begin{enumerate}
	\item $I_l$ is an interval for $1\leq l\leq t$.
	\item $max\left\{i|i\in I_l\right\}+2\leq min\left\{i|i\in I_{l+1}\right\}$ for each $1\leq l\leq t-1$.
	\item $I_j\cap I_k=\emptyset$ if $j\neq k$
\end{enumerate}
Then for any subset $I$ with the decomposition $I=I_1\cup I_2\cup\cdots \cup I_t$, consider the string module
\begin{equation*}
	M_I(w)=\bigoplus^t_{l=1}M(w_{I_l}).
\end{equation*}
Set
\begin{equation*}
	S(w)=\left\{I\in \left\{1,2,\cdots ,n\right\}|M_I(w) \ is \ a \ submodule \ of \ M(w)\right\}.
\end{equation*}
It is clear that each substring induces a canonical submodule of $M(w)$. Moreover, since the decomposition is disjoint with maximal length, the supports of any substring modules is pairwise disjoint. So any subset in $S(w)$ induces a canonical submodule of $M(w)$. The following lemma shows the canonical submodule is equivalent to $S(w)$.
\begin{lemma}
	For any string $w$ in $Q$, there is a bijection 
	\begin{equation*}
		f:S(w)\rightarrow CS(M)
	\end{equation*}
\end{lemma}
\begin{proof}
	We construct the inverse of the map $f$. For any canonical submodule $N\hookrightarrow M(w)$, let
	\begin{equation*}
		I_N=\left\{i|x_i \in N\right\},
	\end{equation*}
	Then if we have $N=N(w_1)\oplus N(w_2)\oplus\cdots\oplus N(w_t)$, then each $w_i$ must be a substring of $w$. So $I_N$ has decomposition $I_1\cup I_2\cup\cdots \cup I_t$ with $w_{I_i}=w_i$ for any $1\leq i\leq t$. Therefore $f(I_N)=N$. $I_{f(I)}=I$ is obviously true for any $I\in S(w)$.
\end{proof}
From now on, we will identify a canonical submodule with its index set. Sometimes, we will use a subset of the index set to represent a string module (which may not necessarily be a submodule) whose support is the subset, provided there is no ambiguity.

Then we can state the following $main\ theorem$: for any string object $r$ and two cluster-tilting object $T,T'$ where $T'=\mu_k(T)$ in $\mathcal{C}$. Let $Ext_{\mathcal{C}}^1(T,r)=M(w)$ and $Ext_{\mathcal{C}}^1(T',r)=M(w')$.
\begin{theorem}
	There exists functions $v,v'$ valued on $\mathbb{Z}$ satisfying the following polynomial property.
	\begin{equation*}
		\sum_{U\in CS(M(w))}q^{\frac{v(U)}{2}X^{ind_T(r)+B_T dim(U)}}=\sum_{U'\in CS(M(w'))}q^{\frac{v'(U')}{2}(X')^{ind_{T'}r+B_{T'} dim(U')}}.
	\end{equation*}
\end{theorem}
As a corollary of this theorem, we get an expression of quantum cluster variables given by string modules.
\begin{theorem}
	Let $r$ be a string object in $\mathcal{C}$ and $T$ be a cluster-tilting object. Then the Laurent expansion of quantum cluster variable corresponding to $r$ with respect to the cluster corresponding to $T$ is
	\begin{equation*}
		X_r=\sum_{N\in CS(M(w))}q^{\frac{v(N)}{2}}X^{ind_T(r)+B_TdimN}.
	\end{equation*}
	where $Ext_{\mathcal{C}}^1(T,r)=M(w)$ and $v:CS(M(w))\rightarrow \mathbb{Z}$.
\end{theorem}
\begin{proof}
	By the main theorem, the right-hand equation is invariant under mutations of cluster-tilting objects. Then by a finite sequence of mutation, we can obtain a cluster-tilting object $T_r$ containing $r$ as a direct summand. Then, $r$ associates with the zero module in $A(T_r)$ and the expression is trivially equal to the cluster variable associate with $X_r$
\end{proof}
\section{PROOF OF THE MAIN THEOREM}In this section, we give a proof of the main theorem. The idea is to cut the arc $r$ into several pieces, such that for each segment, we can define a specific function that makes the identity true. In this way, the equations become a summation of simpler equations, each of which is simply the binomial expansion for the corresponding piece. By collecting the functions for each piece, we will obtain the definitions of $v$ and $v'$.
\subsection{Notations}Let $r$ be an arc. Any two cluster tilting object $T'=\mu_k(T)$ for some $k$. Let $Ext_{\mathcal{C}}^1(T,r)=M(w)$ be the string module of $End(T)$ and $Ext_{\mathcal{C}}^1(T',r)=M(w')$ be the string module of $End(T')$. 
\begin{itemize}
	\item $CS(M(w))$ is the set of canonical modules of $M(w)$ and $CS(M(w'))$ is the set of canonical modules of $M(w')$.
	\item $B_T=(b_{ij})$ is the matrix of quiver of $End(T)$ and $B_{T'}=(b'_{ij})$ is the matrix of quiver of $End(T')$.
	\item The triangulation associating with $T$ is $\Gamma$ and the internal arcs are $\tau_1,\cdots,\tau_n$. Then the triangulation associating with $T'$ is $\Gamma'$ and the internal arcs are $\tau_1,\cdots,\tau_{k-1},\tau_{k'},\tau_{k+1},\cdots,\tau_n$.
	\item The cluster associating $\Gamma$ is $(X_{\tau_1},\cdots,X_{\tau_n})$ which is simply denoted by $(X_1,\cdots,X_n)$. The cluster associating $\Gamma'$ is $(X'_1,\cdots,X'_n)$. For any $a=(a_1,\cdots,a_n)\in \mathbb{N}^n$
	\begin{equation*}
		X^a=q^{-\sum_{i<j}a_ia_j\Lambda_{ij}/2}X_1^{a_1}\cdots X_n^{a_n}.
	\end{equation*} 
	By the rules of mutation, $X'_k=X^{-e_k+\sum_{j}[b_{jk}]_+e_j}+X^{-e_k+\sum_{j}[-b_{jk}]_+e_j}$. Also dually
	\begin{equation*}
		(X')^a=q^{-\sum_{i<j}a_ia_j\Lambda'_{ij}/2}(X'_1)^{a_1}\cdots (X'_n)^{a_n}.
	\end{equation*}
	and $X_k=(X')^{-e_k+\sum_{j}[b'_{jk}]_+e_j}+(X')^{-e_k+\sum_{j}[-b'_{jk}]_+e_j}$.
	\item Let $X^T(U)=X^{ind_T(V)+B_TdimU}$ for $U\in CS(V)$ in $End(T)$ where $V$ is based on context. Similarly $(X')^T(U')=(X')^{ind_{T'}(V')+B_{T'}dimU'}$ for $U'\in CS(V')$ in $End(T')$.
	\item $(b_k)_+=\sum_{j}[b_{jk}]_+e_j$ and $(b_k)_-=\sum_{j}[-b_{jk}]_+e_j$. Then $b_k=(b_k)_+-(b_k)_-$ is the $k-th$ column of $B_T$.
\end{itemize}
\begin{definition}[\cite{Huang2021AnEF}]
	Let $S,S'$ be two sets
	\begin{enumerate}
		\item A subset $\mathcal{B}$ of its power set is called a partition of $S$ if $\cup_{R\in \mathcal{B}}R=S$ and $R\cap R'=\emptyset$ if $R\neq R'\in \mathcal{B}$.
		\item A partition bijection from a set $S$ to $S'$ is a bijection from a partition of $S$ to a partition of $S'$.
	\end{enumerate}
\end{definition}
\begin{remark}
	We will describe such a partition bijection $\varphi$ from $CS(V)$ to $CS(V')$ where $V$ is a string module of $End(T)$ and $V'$ is a string module of $End(T')$. We say the partition bijection satisfies the polynomial property if 
	\begin{equation*}
		\sum_{U\in R}X^T(U)=\sum_{U'\in R'}(X')^T(U').
	\end{equation*}
	where $\varphi(R)=R'$.
\end{remark}
\begin{lemma}[\cite{Dehy2007OnTC}]
	$ind_{T'}(r)=E_{\epsilon}*ind_T(r)$ where $E_\epsilon$ is the matrix 
	\begin{equation*}
		(E_\epsilon)_{ij}=\left\{
		\begin{aligned}
			&\delta_{ij} &\mbox{if $j\neq k$}, \\
			&-1 &\mbox{if $i=j=k$}, \\
			&[\epsilon b_{ik}] & \mbox{if $i=k\neq j$}.
		\end{aligned}
		\right.
	\end{equation*}
	where $\epsilon$ is the sign of $k$-th entry of $ind_T(r)$
\end{lemma}
\begin{lemma}[\cite{Huang2021AnEF}]
	Let $c$ be a positive integer, then
	\begin{align*}
		(X'_k)^{c}=\sum_{\lambda\in \left\{0,1\right\}^c}q^{\frac{n(\lambda)}{2}}X^{-ce_k+(\sum \lambda_i)(b_k)_++(c-\sum \lambda_i)(b_k)_-},
	\end{align*}
	where for any $\lambda=(\lambda_i)$ and $\lambda'=(\lambda'_i)$ satisfy $\lambda_j=\lambda'_j$ for $j\neq l$ and $\lambda_l=1$, $\lambda'_l=0$, we have
	\begin{equation*}
		n(\lambda)-n(\lambda')=c-2l+1.
	\end{equation*}
	and $n(0,0,\cdots,0)=0$.
\end{lemma}
\begin{corollary}
	Let $d=(d_1,\cdots,d_m)\in \mathbb{Z}^m$ with $d_k>0$. 
	\begin{align*}
		(X')^d=\sum_{\lambda \in \left\{0,1\right\}^{d_k}}q^{v_d(\lambda)/2}X^{d-2d_ke_k+(\sum_{i}\lambda_i)(b_k)_++(d_k-\sum_{i}\lambda_i)(b_k)_-}. \\
	\end{align*}
	where $v_d:\left\{0,1\right\}^{d_k}\rightarrow \mathbb{Z}$ satisfies for any $\lambda$ and $\lambda'$ satisfy $\lambda_l=\lambda'_l$ for $j\neq l$ and $\lambda_l=1$, $\lambda'_l=0$, we have
	\begin{equation*}
		v_d(\lambda)-v_d(\lambda')=d_k-2l+1.
	\end{equation*}
	and $v_d(0,0,\cdots,0)=0$.
	\par Dually,
	\begin{equation*}
		X^d=\sum_{\lambda \in \left\{0,1\right\}^{d_k}}q^{v'_d(\lambda)/2}(X')^{d-2d_ke_k+(\sum_{i}\lambda_i)(b_k)_++(d_k-\sum_{i}\lambda_i)(b_k)_-}.
	\end{equation*}
	where $v'_d:\left\{0,1\right\}^{d_k}\rightarrow \mathbb{Z}$ satisfies for any $\lambda$ and $\lambda'$ satisfy $\lambda_l=\lambda'_l$ for $j\neq l$ and $\lambda_l=1$, $\lambda'_l=0$, we have
	\begin{equation*}
		v'_d(\lambda)-v'_d(\lambda')=d_k-2l+1.
	\end{equation*}
	and $v'_d(0,0,\cdots,0)=0$.
\end{corollary}
\begin{proof}
	We just prove the first statement and the other is similar. Since $(X')^d=q^{\frac{\alpha}{2}}(X')^a(X')^{d_ke_k}$ where $\alpha=-\Lambda'(a,d_ke_k)$ and the $k$-th coordinate of $a$ is zero. By above lemma
	\begin{equation*}
		(X')^{d_ke_k}=\sum_{\lambda\in \left\{0,1\right\}^{d_k}}q^{\frac{n(\lambda)}{2}}X^{-d_ke_k+(\sum \lambda_i)(b_k)_++(d_k-\sum \lambda_i)(b_k)_-}.
	\end{equation*}
	Then for any $\lambda\in \left\{0,1\right\}^{d_k}$, define
	\begin{equation*}
		v_d(\lambda)=n(\lambda)+\alpha+\Lambda(a,-d_ke_k+(\sum \lambda_i)(b_k)_++(d_k-\sum \lambda_i)(b_k)_-).
	\end{equation*}
	Then the $v_d$ must make the equation true. For any $\lambda$ and $\lambda'$ satisfy the above condition 
	\begin{align*}
		v_d(\lambda)-v_d(\lambda')&=n(\lambda)-n(\lambda')+\Lambda(a,(b_k)_+-(b_k)_-) \\
		&=n(\lambda)-n(\lambda')+\Lambda(a,b_k) \\
		&=n(\lambda)-n(\lambda').
	\end{align*}
	where the last equation follows by the compatible pair and the fact the $k$-th coordinate of $a$ vanishes.
	\begin{align*}
		v_d(0,\cdots,0)&=n(0,\cdots,0)-\Lambda'(a,d_ke_k)+\Lambda(a,-d_ke_k+d_k(b_k)_-) \\
		&=-d_k\Lambda(a,-e_k+(b_k)_+)+\Lambda(a,-d_ke_k+d_k(b_k)_-) \\
		&=\Lambda(a,-d_kb_k) \\
		&=0.
	\end{align*}
\end{proof}
With a suitable label change we can have the following subquiver in $Q(\Gamma)$ for some $k$
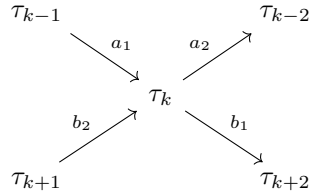
\par \begin{center}
	\begin{tikzcd}
		{\tau_{k-1}} && {\tau_{k-2}} \\
		& {\tau_k} \\
		{\tau_{k+1}} && {\tau_{k+2}}
		\arrow["a_1",from=1-1, to=2-2]
		\arrow["a_2",from=2-2, to=1-3]
		\arrow["b_1",from=2-2, to=3-3]
		\arrow["b_2",from=3-1, to=2-2]
	\end{tikzcd}
\end{center}
where $a_2a_1$ and $b_2b_1$ are in the relation. Note that the vertex may be frozen and we let the arrow is empty arrow in this situation. Then $b_{k,k-1}=b_{k,k+1}=1$ and $b_{k,k-2}=b_{k,k+2}=-1$.
\begin{lemma}[\cite{Brstle2011AMI}]
	Let $M(w)$ have a minimal injective resolution $0\longrightarrow M(w)\longrightarrow I_0\longrightarrow I_1$. Then $(ind_T(r))_j=dimHom(S_j,I_1)-dimHom(S_j,I_0)$ where $S_i$ is the simple module of $End(T)$ corresponding to any vertex $i$.
\end{lemma}
\begin{corollary}
	Let $w$ be a string over $Q(\Gamma)$ 
	\begin{equation*}
		w=x_1\stackrel{\alpha_1}{\longleftrightarrow} x_2\stackrel{\alpha_2}{\longleftrightarrow}\cdots \stackrel{\alpha_{s-1}}{\longleftrightarrow}x_s,
	\end{equation*}
	which don't have a vertex labeled $\left\{\tau_{k-2},\tau_{k-1},\tau_k,\tau_{k+1},\tau_{k+2}\right\}$. Then the $k$-th entry of $ind_T(r)$ is zero.
\end{corollary}
\begin{proof}
	Let $T_1\rightarrow T_0\rightarrow r\rightarrow T_1[1]$. Then the index of $r$ is $[T_0]-[T_1]$. So we have a minimal injective presentation where $F=Ext^1_{\mathcal{C}}(T,-)$
	\begin{equation*}
		0\rightarrow F(r)\rightarrow FT_1[1] \rightarrow FT_0[1].
	\end{equation*}So the $k$-th entry of $ind_T(r)$ is equal to $dimHom(S_k,FT_0[1])-dimHom(S_k,FT_1[1])$. Since $Hom(S_k,F(r))$ and $Ext^1(S_k,F(r))$ are both zero, the $k$-th entry of $ind_T(r)$ is zero.
\end{proof}
\par Since the surface is unpunctured, the quiver can not have loops or $2-cycle$. So we have the following two possibility:
\begin{enumerate}
	\item $\tau_{k-1}\neq \tau_{k+1}$ and $\tau_{k-2}\neq \tau_{k+2}$.
	\item only one of them is true: $\tau_{k-1}=\tau_{k+1}$ and $\tau_{k-2}=\tau_{k+2}$.
\end{enumerate}
because $\tau_{k-1}=\tau_{k+1}$ and $\tau_{k-2}=\tau_{k+2}$ both happen only if the surface is a torus with one marked point\cite{Huang2021AnEF}.
\subsection{When $\tau_{k-1}\neq \tau_{k+1}$ and $\tau_{k-2}\neq \tau_{k+2}$} \
\par In this situation, we can divide the string $w$ into different subword such that each subword $N$ does not contain any element in $H=\left\{\tau_{k-2},\tau_{k-1},\tau_k,\tau_{k+1},\tau_{k+2}\right\}$ as its vertex or all the vertices are one of them. The latter must be one of the following form
\begin{enumerate}
	\item $x_{i-1}\longrightarrow x_i \longrightarrow x_{i+1}$ or $x_{i-1}\longleftarrow x_i \longleftarrow x_{i+1}$ with $x_i=\tau_k$.
	\item $x_{i-1}\longrightarrow x_i \longleftarrow x_{i+1}$ or $x_{i-1}\longleftarrow x_i \longrightarrow x_{i+1}$ with $x_i=\tau_k$.
	\item $\tau_{k-2}\longrightarrow \tau_{k-1}$ or $\tau_{k+1}\longrightarrow \tau_{k+2}$ or their inverse.
	\item $x_1\longleftrightarrow x_2$ with $x_1=\tau_k$ or $x_{s-1}\longleftrightarrow x_s$ with $x_s=\tau_k$.
	\item $x_1\longleftrightarrow x_2$ with $x_1\in \left\{\tau_{k-2},\tau_{k-1},\tau_{k+1},\tau_{k+2}\right\}$ and $x_2\notin H$.
	\item $x_{n-1}\longleftrightarrow x_n$ with $x_n\in \left\{\tau_{k-2},\tau_{k-1},\tau_{k+1},\tau_{k+2}\right\}$ and $x_{n-1}\notin H$.
\end{enumerate}
Note that we always identify two inverse string. We analysis these cases in the following way:
\par Case(1). Assume we have a subword $N:\tau_{k-1}\longrightarrow \tau_k\longrightarrow \tau_{k+2}$. The other possibilities just change the label or reverse all the arrow and we can have similar property using the same methods. Then let $N'=\mu_k(N):\tau_{k-1}\longrightarrow \tau_{k+2}$ and there is a partition bijection $\varphi_N$ from $CS(N)$ to $CS(N')$
\begin{align*}
	0\longrightarrow 0\longrightarrow 0&\longleftrightarrow 0\longrightarrow 0, \\
	\begin{aligned}
		\left\{\begin{aligned}
			0\longrightarrow 0\longrightarrow \tau_{k+2} \\
			0\longrightarrow \tau_k\longrightarrow \tau_{k+2}
		\end{aligned}  
		\right\}
	\end{aligned} &\longleftrightarrow 0\longrightarrow \tau_{k+2},\\
	\tau_{k-1}\longrightarrow \tau_k\longrightarrow \tau_{k+2}&\longleftrightarrow \tau_{k-1}\longrightarrow \tau_{k+2}. \\
\end{align*}
Since $Hom(S_k,N)=0$ and $Ext^1(S_k,N)=0$, then the $k$-th entry of index of $N$: $(ind_T(N))_k=0$. In this situation, $ind_{T'}(N')=ind_T(N)$. In addition, the partition bijection has the polynomial property:
\begin{align*}
	X^T(0) &=(X')^T{(0)}, \\
	X^T(e_{k+2})+X^T(e_k+e_{k+2})&=(X')^T(e_{k+2}), \\
	X^T(e_{k-1}+e_k+e_{k+2}) &=(X')^T(e_{k-1}+e_{k+2}).
\end{align*}
Here we use the dimension vector to represent the submoudles on both side.
\par Case(2). Assume we have a subword $N:\tau_{k-1}\longrightarrow \tau_k\longleftarrow \tau_{k+1}$. The other possibilities just change the label or reverse all the arrow and we can have similar property using the same methods. Then let $N'=\mu_k(N):\tau_{k-1}\longleftarrow \tau_k'\longrightarrow \tau_{k+1}$ and there is a partition bijection $\varphi_N$ from $CS(N)$ to $CS(N')$
\begin{align*}
	\begin{aligned}
		\left\{\begin{aligned}
			0\longrightarrow 0\longleftarrow 0 \\
			0\longrightarrow \tau_k \longleftarrow 0
		\end{aligned}  
		\right\}
	\end{aligned} &\longleftrightarrow 0\longleftarrow 0\longrightarrow 0,\\
	0\longrightarrow \tau_k\longleftarrow \tau_{k+1}&\longleftrightarrow 0\longleftarrow 0\longrightarrow \tau_{k+1},\\
	\tau_{k-1}\longrightarrow \tau_k\longleftarrow 0&\longleftrightarrow \tau_{k-1}\longleftarrow 0\longrightarrow 0,\\
	\tau_{k-1}\longrightarrow \tau_k\longleftarrow \tau_{k+1} &\longleftrightarrow \begin{aligned}
		 \left\{\begin{aligned}
			\tau_{k-1}\longleftarrow \tau_{k'}\longrightarrow \tau_{k+1}\\
			\tau_{k-1}\longleftarrow 0\longrightarrow \tau_{k+1}
		\end{aligned}  
		\right\}
	\end{aligned}. 
\end{align*}
Since $dimHom(S_k,N)=1$ and $dimExt^1(S_k,N)=0$, then the $k$-th entry of index of $N$: $(ind_T(N))_k=-1$. In this situation, $(ind_{T'}(N'))_j=(ind_T(N))_j-[-b_{jk}]_+$ for $j\neq k$ and $(ind_{T'}(N'))_k=1$. In addition, the partition bijection has the polynomial property:
\begin{align*}
	X^T(0)+X^T(e_k) &=(X')^T(0), \\
	X^T(e_k+e_{k+1}) &=(X')^T(e_{k+1}), \\
	X^T(e_k+e_{k-1}) &=(X')^T(e_{k-1}), \\
	X^T(e_{k-1}+e_k+e_{k+1})&=(X')^T(e_{k-1}+e_k+e_{k+1})+(X')^T(e_{k-1}+e_{k+1}).
\end{align*}
Here we use the dimension vector to represent the submoudles on both side.
\par Case(3). Assume we have a subword $N:\tau_{k-2}\longrightarrow \tau_{k-1}$. Then let $N'=\mu_k(N):\tau_{k-2}\longrightarrow \tau_k'\longrightarrow \tau_{k-1}$. Then this case is dual to Case(1): there is a partition bijection $\varphi_N$ from $CS(N)$ to $CS(N')$ which is the inverse of the partition bijection in Case(1).
\par Case(4). Assume we have a subword $\tau_k \longrightarrow \tau_{k-2}$. The other possibilities just change the label or reverse all the arrow and we can have similar property using the same methods. Then let $N'=\mu_k(N):\tau_{k-2}$ and there is a partition bijection $\varphi_N$ from $CS(N)$ to $CS(N')$
\begin{align*}
	0\longrightarrow 0 &\longleftrightarrow 0, \\
	\left\{0\longrightarrow \tau_{k-2},\tau_k \longrightarrow \tau_{k-2}\right\} &\longleftrightarrow \tau_{k-2}.
\end{align*}
Since $dimHom(S_k,N)=0$ and $dimExt^1(S_k,N)=0$, then the $k$-th entry of index of $N$: $(ind_T(N))_k=0$. In this situation, $ind_{T'}(N')=ind_T(N)$. In addition, the partition bijection has the polynomial property:
\begin{align*}
	X^T(0)&=(X')^T(0), \\
	X^T(e_{k-2})+X^T(e_k+e_{k-2})&=(X')^T(e_{k-2}).
\end{align*}
Here we use the dimension vector to represent the submoudles on both side.
\par Case(5). Assume we have a subword $N:\tau_{k-2}$. Then let $N'=\mu_k(N):\tau'_{k}\longleftarrow \tau_{k-2}$. Then this case is dual to Case(4): there is a partition bijection $\varphi_N$ from $CS(N)$ to $CS(N')$ which is the inverse of the partition bijection in Case(4).
\par Case(6) is dual to Case(5).
\subsection{When $\tau_{k-1}=\tau_{k+1}$ and $\tau_{k-2}\neq \tau_{k+2}$} \
\par The other situation($\tau_{k+2}=\tau_{k-2}$) is dual and we just focus on this one. Like above, we  we can divide the string $w$ into different subword such that each subword $N$ does not contain any element in $H=\left\{\tau_{k-2},\tau_{k-1},\tau_k,\tau_{k+1},\tau_{k+2}\right\}$ as its vertex or all the vertices are one of them. In this situation, the difference with above is there is a finite cyclic subword. The latter is one of the following form
\begin{enumerate}
	\item $\tau_{k-2}\longleftarrow \tau_k \longrightarrow \tau_{k+2}$ or its inverse.
	\item $\tau_{k-2}\longrightarrow \tau_{k-1} \longrightarrow \tau_k\longrightarrow \tau_{k-2}$ or $\tau_{k+2}\longrightarrow \tau_{k-1} \longrightarrow \tau_k\longrightarrow \tau_{k+2}$ or their inverse.
	\item $x_1\longleftrightarrow x_2$ with $x_1=\tau_k$ and $x_2\neq \tau_{k-1}$ or $x_{s-1}\longleftrightarrow x_s$ with $x_s=\tau_k$ and $x_{k-1}\neq \tau_{k-1}$.
	\item $x_1\longleftrightarrow x_2$ with $x_1\in \left\{\tau_{k-2},\tau_{k-1},\tau_{k+1},\tau_{k+2}\right\}$ and $x_2\notin H$.
	\item $x_{n-1}\longleftrightarrow x_n$ with $x_n\in \left\{\tau_{k-2},\tau_{k-1},\tau_{k+1},\tau_{k+2}\right\}$ and $x_{n-1}\notin H$.
	\item $\tau_{k-1}\rightarrow \tau_k \leftarrow \tau_{k-1}\rightarrow \cdots \tau_k\leftarrow \tau_{k-1}$.
	\item $\tau_{k+2}\rightarrow\tau_{k-1}\rightarrow \tau_k \leftarrow \tau_{k-1}\rightarrow \cdots \tau_k\leftarrow \tau_{k-1}$ or $\tau_{k-1}\rightarrow \tau_k \leftarrow \tau_{k-1}\rightarrow \cdots \tau_k\leftarrow \tau_{k-1}\leftarrow \tau_{k+2}$.
	\item $\tau_{k+2}\rightarrow\tau_{k-1}\rightarrow \tau_k \leftarrow \tau_{k-1}\rightarrow \cdots \tau_k\leftarrow \tau_{k-1}\leftarrow \tau_{k-2}$.
\end{enumerate}
Note that we just list one of the possibility in Case(7),(8) since $\tau_{k+2}$ and $\tau_{k-2}$ have the same position. According to the rules how we divide, if $\tau_{k-1}$ is the start or end point of subword, then it must be the start or end point of the string. Since $\tau_{k-2}$ and $\tau_{k+2}$ are equally important, by exchanging them and dual, the above situations are all the possibility. We just choose one representative for each case.
\par Case(1). This is the same as Case(2) in Section5.2.
\par Case(2). Assume we have a subword $N:\tau_{k-2}\longrightarrow \tau_{k-1} \longrightarrow \tau_k\longrightarrow \tau_{k-2}$. The other possibilities just change the label or reverse all the arrow and we can have similar property using the same methods. Then let $N'=\mu_k(N):\tau_{k-2}\longrightarrow \tau_{k}' \longrightarrow \tau_{k-1}\longrightarrow \tau_{k-2}$ and there is a partition bijection $\varphi_N$ from $CS(N)$ to $CS(N')$.
\begin{align*}
	0\rightarrow0\rightarrow0\rightarrow0 &\longleftrightarrow 0\rightarrow0\rightarrow0\rightarrow0 ,\\
	\begin{aligned}
		\left\{\begin{aligned}
			0\rightarrow0\rightarrow0\rightarrow \tau_{k-2} \\
			0\rightarrow0\rightarrow \tau_k\rightarrow \tau_{k-2}
		\end{aligned}
		\right\}
	\end{aligned}&\longleftrightarrow 0\rightarrow0\rightarrow0\rightarrow \tau_{k-2} ,\\
	0\rightarrow \tau_{k-1}\rightarrow \tau_k\rightarrow \tau_{k-2} &\longleftrightarrow \begin{aligned}
		\left\{\begin{aligned}
			0\rightarrow0\rightarrow \tau_{k-1}\rightarrow \tau_{k-2} \\
			0\rightarrow \tau_k'\rightarrow \tau_{k-1}\rightarrow \tau_{k-2}
		\end{aligned}
		\right\}
	\end{aligned},\\
	\tau_{k-2}\rightarrow \tau_{k-1}\rightarrow \tau_k\rightarrow \tau_{k-2} &\longleftrightarrow\tau_{k-2}\longrightarrow \tau_{k}' \longrightarrow \tau_{k-1}\longrightarrow \tau_{k-2}.
\end{align*}
Since $dimHom(S_k,N)=0$ and $dimExt^1(S_k,N)=0$, then the $k$-th entry of index of $N$: $(ind_T(N))_k=0$. In this situation, $ind_{T'}(N')=ind_T(N)$. In addition, the partition bijection has the polynomial property:
\begin{align*}
	X^T(0)&=(X')^T(0), \\
	X^T(e_{k-2})+X^T(e_k+e_{k-2})&=(X')^T(e_{k-2}), \\
	X^T(e_{k-1}+e_k+e_{k-2})&=(X')^T(e_{k-1}+e_{k-2}) +(X')^T(e_k+e_{k-1}+e_{k-2}),\\
	X^T(e_{k-2}+e_{k-1}+e_k+e_{k-2})&=(X')^T(e_{k-2}+e_{k-1}+e_k+e_{k-2}).
\end{align*}
Here we use the dimension vector to represent the submoudles on both side.
\par Case(3),(4),(5) are just the same as Section5.2.
\par Case(6). Now we have a subword $N:\tau_{k-1}\rightarrow \tau_k \leftarrow \tau_{k-1}\rightarrow \cdots \tau_k\leftarrow \tau_{k-1}$ and $\tau_k$ occurs $s$ times in the word. Then let $N'=\mu_k(N):\tau_k'\rightarrow\tau_{k-1}\leftarrow \tau_k' \rightarrow \tau_{k-1}\leftarrow \cdots \tau_{k-1}\leftarrow \tau_k'$ where $\tau_k'$ occurs $s+2$ times. Note that $\tau_{k-1}$ always occurs $s+1$ times. Then we have $dimHom(S_k,N)=s$, $dimExt^1(S_k,N)=0$, $dimHom(S_k,N')=0$, $dimExt^1(S_k,N')=s$. So $(ind_T(N))_k=-s$ and $(ind_{T'}(N'))_k=s$ and
\begin{equation*}
	(ind_{T'}(N'))_j=(ind_T(N))_j-s[-b_{jk}]_+.
\end{equation*}
In order to give a partition bijection $\varphi_N$ from $CS(N)$ to $CS(N')$, we need the following definition:
\begin{definition}
	For $\lambda=(\lambda_1,\cdots,\lambda_{s+1})\in \left\{0,1\right\}^{s+1}$, let $N_\lambda$ be the set of canonical submodule of $N$ satisfying the following condition: if $\lambda_i=1$, then $i$-th $\tau_{k-1}$ appear as a vertex of submodules, otherwise not. Dually, let $N'_\lambda$ be the set of canonical submodule of $N'$ satisfying the following condition: if $\lambda_i=1$, then $i$-th $\tau_{k-1}$ appear as a vertex of submodules, otherwise not. We use 0 to represent $(0,\cdots,0)$ and 1 to represent $(1,\cdots,1)$. For any $\lambda=(\lambda_1,\cdots,\lambda_{s+1})\in \left\{0,1\right\}^{s+1}$, let $1-\lambda=(1-\lambda_1,\cdots,1-\lambda_{s+1})$.
\end{definition}
Apparently, $CS(N)=\bigcup_{\lambda}N_\lambda$ and $CS(N')=\bigcup_{\lambda}N'_\lambda$. In $N_\lambda$, the difference of submodules is how many $\tau_k$ are contained. We say the $i$-th $\tau_k$ is free vertex in $N_\lambda$ if there are one module containing it and another module not containing it in $N_\lambda$ and dually $i$-th $\tau'_k$ is free vertex in $N'_\lambda$ if there are one module containing it and another module not containing it in $N'_\lambda$.
\begin{lemma}
	If $\lambda_i=\lambda_{i+1}=0$ for some $i=1,2,\cdots,s$, then the $i$-th $\tau_k$ is a free vertex of $N_\lambda$. Otherwise it must be a vertex in $N_\lambda$. Dually, If $\lambda_i=\lambda_{i+1}=1$ for some $i=1,2,\cdots,s$, then the $i+1$-th $\tau'_k$ is a free vertex of $N'_\lambda$. The first $\tau'_k$ is a free vertex if and only if $\lambda_1=1$ and the last $\tau'_k$ is a free vertex if and only if $\lambda_{s+1}=1$.
\end{lemma}
\begin{proof}
	Note that as a submodule of $\tau_{k-1}\rightarrow \tau_k\leftarrow \tau_{k-1}$, if there is a $\tau_{k-1}$ as a vertex, then $\tau_k$ must be a vertex. Otherwise $\tau_k$ is free. The dual case is similar.
\end{proof}
\begin{corollary}
	For any $\lambda=(\lambda_1,\cdots,\lambda_{s+1})\in \left\{0,1\right\}^{s+1}$. Let
	\begin{equation*}
		\begin{aligned}
			a_\lambda&=\#\left\{i:\lambda_i=\lambda_{i+1}=0,i=1,2,\cdots,s\right\}, \\
			b_\lambda&=\#\left\{i:\lambda_i=\lambda_{i+1}=1,i=1,2,\cdots,s\right\}+\delta_{\lambda_1,1}+\delta_{\lambda_{s+1},1}.
		\end{aligned}
	\end{equation*}
	Then there are $a_\lambda$ free vertices in $N_\lambda$ and $b_\lambda$ free vertices in $N'_\lambda$.
\end{corollary} 
\begin{lemma}
	For any $\lambda=(\lambda_1,\cdots,\lambda_{s+1})\in \left\{0,1\right\}^{s+1}$. Then
	\begin{equation*}
		b_\lambda-a_\lambda=-s+2(\sum_{i=1}^{s+1}\lambda_i).
	\end{equation*}
\end{lemma}
\begin{proof}
	Let $d=\sum_{i=1}^{s+1}\lambda_i$. We prove this by the following property: if we exchange $\lambda_i$ and $\lambda_{i+1}$ to get a new $\lambda'$. Then $b_\lambda-a_\lambda=b_{\lambda'}-a_{\lambda'}$. This can be proved case by case. For example, if we have $(\lambda_{i-1},\lambda_i,\lambda_{i+1},\lambda_{i+2})=(1,1,0,0)$, then $b_\lambda=1,a_\lambda=1,b_{\lambda'}=0,a_{\lambda'}=0$ and in this case the property is true. The other case is similar. So we can always do this exchange until we have the first $\lambda_1=\cdots=\lambda_d=1$. In this way, $a_\lambda=s-d,b_\lambda=d$ and $b_\lambda-a_\lambda=-s+2d$.
\end{proof}
In summary, we proved that $N_\lambda$ has $a_\lambda$ free vertex, $N'_\lambda$ has $b_\lambda$ free vertex and $b_\lambda-a_\lambda=-s+2(\sum_{i=1}^{s+1}\lambda_i)$. We order the free vertices from left to right in both modules. So for $\lambda=(\lambda_1,\cdots,\lambda_{s+1})\in \left\{0,1\right\}^{s+1}$, we have a disjoint union
\begin{equation*}
	N_\lambda=\bigcup_{u\in \left\{0,1\right\}^{a_\lambda}} N_{\lambda,u},
\end{equation*}
where $N_{\lambda,u}$ is the submodule with $i$-th $\tau_k$ as a vertex if $u_i=1$ and otherwise not. Similarly
\begin{equation*}
	N'_\lambda=\bigcup_{u'\in \left\{0,1\right\}^{b_\lambda}} N'_{\lambda,u'}.
\end{equation*}
If $c_\lambda=-s+2(\sum_{i=1}^{s+1}\lambda_i)$ is positive then we can say $u$ is a element of $\left\{0,1\right\}^{b_\lambda}$ by embedding in the first $a_\lambda$ positions. Then we associate each $N_{\lambda,u}$ with a set $\varphi_N(N_{\lambda,u})$ of submodules $N'_{\lambda,u'}$ where the first $a_\lambda$ coordinates of $u'$ is $1-u$. Obviously there are $2^{c_\lambda}$ element in $\varphi_N(N_{\lambda,u})$. 
\begin{lemma}
	For any $\lambda$ with $c_\lambda>0$ and $u\in \left\{0,1\right\}^{a_\lambda}$, there exists $m_u:\left\{0,1\right\}^{c_\lambda}\rightarrow \mathbb{Z}$ such that
	\begin{equation*}
		X^T(N_{\lambda,u})=\sum_{u'}q^{\frac{m_u(u')}{2}}(X')^T(N'_{\lambda,u'}).
	\end{equation*}
	where $u'$ ranges over $\left\{0,1\right\}^{b_\lambda}$ satisfies the first $a_\lambda$ positions is $1-u$ and be seen as an element in $\left\{0,1\right\}^{c_\lambda}$.
\end{lemma}
\begin{proof}
	First we have 
	\begin{equation*}
		X^T(N_{\lambda,u})=q^aX^T(N_{\lambda,w})*X^{(-\sum w_i+\sum u_i)B_T e_k}.
	\end{equation*}
	and
	\begin{equation*}
		(X')^{T}(N'_{\lambda,u'})=q^b(X')^{T}(N'_{\lambda,w'})*(X')^{(-\sum{w'_i}+\sum u'_i)B_{T'}e_k}.
	\end{equation*}
	for some $a,b$ where the first $a_\lambda$ positions of $u',w'$ is $1-u,1-w$ and the rest positions have the same value. Note that $X^{B_Te_k}=(X')^{-B_{T'}e_k}$. If we have
	\begin{equation*}
		X^T(N_{\lambda,0})=\sum_{u'}q^{\frac{m_0(u')}{2}}(X')^T(N'_{\lambda,u'}),
	\end{equation*}
	where the first $a_\lambda$ positions of $u'$ is $1$.
	So
	\begin{align*}
		X^T(N_{\lambda,u})&=q^aX^T(N_{\lambda,0})*X^{(\sum u_i)B_T e_k} \\
		&=q^a\sum_{u'}q^{\frac{m_0(u')}{2}}(X')^T(N'_{\lambda,u'})*X^{(\sum u_i)B_T e_k} \\
		&=q^a\sum_{u'}q^{\frac{m_0(u')}{2}}(X')^T(N'_{\lambda,u'})*(X')^{-(\sum u_i)B_T e_k} \\
		&=\sum_{w'}q^{\frac{m_u(w')}{2}}(X')^T(N'_{\lambda,w'}).
	\end{align*}
	where the first $a_\lambda$ positions of $w'$ is $1-u$ and 
	\begin{equation*}
		m_u(u')=a+m_0(u')+\Lambda'(N'_{\lambda,u'},-(\sum u_i)B_T e_k).
	\end{equation*}
	So we only need to prove the existence when $u=0$. By definition, we know that the right side of this equation ranges over $\varphi_N(N_{\lambda,u})$. The dimension vector is given by the choice of free vertex. Each free vertex can give a dimension vector $e_k$ or not depending on the choice. So the right side is
	\begin{equation*}
		\sum_{h\in \left\{0,1\right\}^{c_\lambda}}q^{m_0(h)}(X')^{ind_{T'}(N')+(\sum \lambda_i)B_{T'}e_{k-1}+(a_\lambda+\sum h_i)B_{T'}e_k}.
	\end{equation*}
	The left side: $X^T(N_{\lambda,0})$ is 
	\begin{equation*}
		\begin{aligned}
			X^{ind_T(N)+(\sum \lambda_i)B_Te_{k-1}+(s-a_\lambda)B_T e_k}.  \\
		\end{aligned}
	\end{equation*}
	Let
	\begin{equation*}
		d=ind_T(N)+(\sum \lambda_i)B_T e_{k-1}+(s-a_\lambda)B_T e_k.
	\end{equation*}
	The left side can be expressed as the following by corollary 5.5
	\begin{equation*}
		\sum_{h\in \left\{0,1\right\}^{c_\lambda}}q^{v'_d(h)}(X')^{d-2c_\lambda e_k+(\sum h_i)(2e_{k-1})+(c_\lambda-(\sum h_i))(e_{k-2}+e_{k+2})}. \\
	\end{equation*}
	So to finish the proof, we only need to check
	\begin{equation*}
		\begin{aligned}
			&ind_{T'}(N')+(\sum \lambda_i)B_{T'}e_{k-1}+(a_\lambda+\sum h_i)B_{T'}e_k= \\
			&d-2c_\lambda e_k+(\sum h_i)(2e_{k-1})+(c_\lambda-(\sum h_i))(e_{k-2}+e_{k+2}).
		\end{aligned}
	\end{equation*}
	We can consider the $j$-th entry.
	\begin{enumerate}
		\item $j=k$: The left side is $s-2\sum \lambda_i$, and the right side is $-s+2(\sum \lambda_i)-2c_\lambda$. They are equal following by $c_\lambda=2\sum \lambda_i-s$.
		\item $j\neq k$: The left side is 
		\begin{equation*}
			(ind_{T'}(N'))_j+(\sum \lambda_i)b'_{j,k-1}+(a_\lambda+\sum h_i)b'_{jk}.
		\end{equation*}
		The right side is 
		\begin{equation*}
			(ind_T(N))_j+(\sum \lambda_i)b_{j,k-1}+(s-a_\lambda)b_{jk}+(\sum h_i)2\delta_{j,k-1}+(c_\lambda-\sum h_i)(\delta_{j,k-2}+\delta_{j,k+2}).
		\end{equation*}
	\end{enumerate}
	\par They are equal by the following identities
	\begin{equation*}
		\begin{aligned}
			b'_{jk}&=2\delta_{j,k-1}-\delta_{j,k-2}-\delta_{j,k+2}, \\
			(ind_{T'}(N'))_j&=(ind_T(N))_j-s[-b_{jk}]_+ ,\\
			b'_{j,k-1}-b_{j,k-1}&=2[b_{jk}]_+.
		\end{aligned}
	\end{equation*}
	Then we let $m_0(h)=v'_d(h)$ for any $h\in \left\{0,1\right\}^{c_\lambda}$.
\end{proof}
Dually, for $\lambda=(\lambda_1,\cdots,\lambda_{s+1})\in \left\{0,1\right\}^{s+1}$ with $-d_\lambda=-s+2(\sum_{i=1}^{s+1}\lambda_i)<0$. Then the disjoint union still exists and in this case, $a_\lambda=b_\lambda+d_\lambda$. So we can view an element in $\left\{0,1\right\}^{b_\lambda}$ as an element in $\left\{0,1\right\}^{a_\lambda}$ by identify the first $b_\lambda$ coordinate. Then we associate each $N'_{\lambda,u'}$ with a set $\phi_{N}(N'_{\lambda,u'})$ of submodules $N_{\lambda,u}$ where the first $b_\lambda$ coordinates of $u$ is $1-u'$.
\begin{lemma}
	For any $u'\in \left\{0,1\right\}^{b_\lambda}$ with $d_\lambda>0$, there exists $m_{u'}:\left\{0,1\right\}^{d_\lambda}\rightarrow \mathbb{Z}$
	\begin{equation*}
		(X')^{T}(N'_{\lambda,u'})=\sum_{u}q^{\frac{m_{u'}(u)}{2}}X^T(N_{\lambda,u}).
	\end{equation*}
	where $u$ ranges over $\left\{0,1\right\}^{a_\lambda}$ satisfying the first $b_\lambda$ positions is $1-u'$ and be seen as an element in $\left\{0,1\right\}^{d_\lambda}$. 
\end{lemma}
\begin{proof}
	First we can assume $u'=0$ like above. By definition, we know that the right side of this equation ranges over $\phi_N(N'_{\lambda,u'})$. The dimension vector is given by the choice of free vertex. Each free vertex can give a dimension vector $e_k$ or not depending on the choice. So the right side is
	\begin{equation*}
		\sum_{h\in \left\{0,1\right\}^{d_\lambda}}q^{m_0(h)}X^{ind_{T}(N)+(\sum \lambda_i)B_{T}e_{k-1}+(b_\lambda+s-a_\lambda+\sum h_i)B_{T}e_k}.
	\end{equation*}
	The left side is 
	\begin{equation*}
		(X')^{ind_{T'}(N')+(\sum \lambda_i)B_{T'}e_{k-1}}.
	\end{equation*}
	Let 
	\begin{equation*}
		d=ind_{T'}(N')+(\sum \lambda_i)B_{T'}e_{k-1}.
	\end{equation*}
	The left side can be expressed as the following
	\begin{equation*}
		\sum_{h\in \left\{0,1\right\}^{d_\lambda}}q^{v_d(h)}X^{d-2d_\lambda e_k+(d_\lambda-\sum h_i)(2e_{k-1})+(\sum h_i)(e_{k-2}+e_{k+2})}.
	\end{equation*}
	So to finish the proof, we only need to check
	\begin{equation*}
		\begin{aligned}
			&ind_{T}(N)+(\sum \lambda_i)B_{T}e_{k-1}+(b_\lambda+s-a_\lambda+\sum h_i)B_{T}e_k= \\
			&d-2d_\lambda e_k+(d_\lambda-\sum h_i)(2e_{k-1})+(\sum h_i)(e_{k-2}+e_{k+2}).
		\end{aligned}
	\end{equation*}
	We can consider the $j$-th entry.
	\begin{enumerate}
		\item $j=k$: The left side is $-s+2\sum \lambda_i$, and the right side is $s-2(\sum \lambda_i)-2d_\lambda$. They are equal following by $d_\lambda=-2\sum \lambda_i+s$.
		\item $j\neq k$: The left side is 
		\begin{equation*}
			(ind_{T}(N))_j+(\sum \lambda_i)b_{j,k-1}+(2(\sum \lambda_i+\sum h_i)b_{jk}.
		\end{equation*}
		The right side is 
		\begin{equation*}
			(ind_{T'}(N'))_j+(\sum \lambda_i)b'_{j,k-1}+(d_\lambda-\sum h_i)2\delta_{j,k-1}+(\sum h_i)(\delta_{j,k-2}+\delta_{j,k+2}).
		\end{equation*}
	\end{enumerate}
	\par They are equal by the following identities
	\begin{equation*}
		\begin{aligned}
			(ind_{T'}(N'))_j&=(ind_T(N))_j-s[-b_{jk}]_+ ,\\
			b'_{j,k-1}-b_{j,k-1}&=2[b_{jk}]_+.
		\end{aligned}
	\end{equation*}
	Then we let $m_0(h)=v'_d(h)$ for any $h\in \left\{0,1\right\}^{c_\lambda}$.
\end{proof}
The remain situation is that if we have $\lambda=(\lambda_1,\cdots,\lambda_{s+1})\in \left\{0,1\right\}^{s+1}$ with $-s+2(\sum_{i=1}^{s+1}\lambda_i)=0$. Then we have $a_\lambda$ submodules of $N_\lambda$ and $b_\lambda$ submodules of $N'_{\lambda}$ and in this case, $a_\lambda=b_\lambda$. The disjoint union still exists and we associate each $N_{\lambda,u}$ with $N'_{\lambda,1-u}$. And we can get the following equation with similar strategy
\begin{equation*}
	X^T(N_{\lambda,u})=(X')^{T}(N'_{\lambda,1-u}).
\end{equation*}
In fact, when $u=0$ the left is equal to
\begin{equation*}
	X^{ind_T(N)+(\sum \lambda_i)B_Te_{k-1}+(s-a_\lambda)B_Te_k}.
\end{equation*}
and the right side is
\begin{equation*}
	(X')^{ind_{T'}(N')+(\sum \lambda_i)B_{T'}e_{k-1}+b_\lambda B_{T'}e_k}.
\end{equation*}
And they are equal by the same calculation.
\par In summary, we have a partition bijection between $CS(N)$ and $CS(N')$. In addition, this bijection always satisfies the polynomial property.
\par Case(7). Assume we have $N:\tau_{k+2}\rightarrow\tau_{k-1}\rightarrow \tau_k \leftarrow \tau_{k-1}\rightarrow \cdots \tau_k\leftarrow \tau_{k-1}$ and $\tau_k$ occurs $s$ times. The other case is nothing but changing the label. Let $N'=\mu_k(N):\tau_{k+2}\rightarrow \tau_k'\rightarrow\tau_{k-1}\leftarrow \tau_k' \rightarrow \tau_{k-1}\leftarrow \cdots \tau_{k-1}\leftarrow \tau_k'$. The only difference between (6) and (7) is the first vertex labeled $\tau_{k-2}$. Let $N^0$(or $N'^0$) the set of submodules of $N$(or $N'$) without the vertex $\tau_{k+2}$ and $N^1$(or $N'^1$) the set of submodules of $N$(or $N'$) with the vertex $\tau_{k+2}$. Then we have the following disjoint union: 
\begin{align*}
	CS(N)=N^0\cup N^1 ,\\
	CS(N')=N'^0\cup N'^1.
\end{align*}
\par If we identify a submodule of $N$ without the vertex $\tau_{k+2}$ with a submodule of $\tau_{k-1}\rightarrow \tau_k \leftarrow \tau_{k-1}\rightarrow \cdots \tau_k\leftarrow \tau_{k-1}$ and a submodule of $N'$ without the vertex $\tau_{k+2}$ with a submodule of $\tau_k'\rightarrow\tau_{k-1}\leftarrow \tau_k' \rightarrow \tau_{k-1}\leftarrow \cdots \tau_{k-1}\leftarrow \tau_k'$, then we have the same result in case(6). In other words, we can see $N^0$ as $CS(N)$ in case (6) and $N'^0$ as $CS(N')$ in case (6). So there is a partition bijection between $N^0$ and $N'^0$ which has the polynomial property. Then we want to get a partition bijection between submodules of $N$ with the vertex $\tau_{k+2}$ and submodules of $N'$ with the vertex $\tau_{k+2}$. We can have the following observation
\begin{lemma}
	The submodule of $N$ with the vertex $\tau_{k+2}$ can be identified with the submodule of $N$ in case(6) where $\tau_k$ occurs $s-1$ times. Similarly, the submodule of $N'$ with the vertex $\tau_{k+2}$ can be identified with the submodule of $N'$ in case(6) where $\tau_k$ occurs $s-1$ times.
\end{lemma}
\begin{proof}
	We just illustrate the first part and the later can be followed by the same. If the submodule of $N$ has $\tau_{k+1}$, then it contains $\tau_{k+2}\rightarrow\tau_{k-1} \rightarrow\tau_k$ as a subword. The other vertex if free. So we can identify them with a submodule of $\tau_{k-1}\rightarrow \tau_k \leftarrow \tau_{k-1}\rightarrow \cdots \tau_k\leftarrow \tau_{k-1}$ where $\tau_k$ occurs $s-1$ times.
\end{proof}
\par However, things are not done because in this case the polynomial property is not trivial. We need to investigate the dimension vector. Let
\begin{equation*}
	N^1=\bigcup_{\lambda\in \left\{0,1\right\}^{s}}N^1_\lambda.
\end{equation*}
Similarly
\begin{equation*}
	N'^1=\bigcup_{\lambda\in \left\{0,1\right\}^{s}}N'^1_\lambda.
\end{equation*}
By case(6), we know there is a partition bijection between $N^1_\lambda$ and $N'^1_\lambda$ for any $\lambda \in \left\{0,1\right\}^{s}$. We just need to prove here the polynomials are equal. We only prove when $\lambda$ satisfies $c_\lambda:-(s-1)+2\sum \lambda_i>0$ and the other case is the same. In this way, $a_\lambda+c_\lambda=b_\lambda$. Then
\begin{equation*}
	N^1_\lambda=\bigcup_{u\in \left\{0,1\right\}^{a_\lambda}} N^1_{\lambda,u}.
\end{equation*}
\begin{equation*}
	N'^1_\lambda=\bigcup_{u'\in \left\{0,1\right\}^{b_\lambda}} N'^1_{\lambda,u'}.
\end{equation*}
\begin{lemma}
	For any $\lambda$ with $c_\lambda>0$ and $u\in \left\{0,1\right\}^{a_\lambda}$, there exists $m_u:\left\{0,1\right\}^{c_\lambda}\rightarrow \mathbb{Z}$
	\begin{equation*}
		X^T(N^1_{\lambda,u})=\sum_{u'}q^{\frac{m_u(u')}{2}}(X')^T(N'^1_{\lambda,u'}).
	\end{equation*}
	where $u'$ ranges over $\left\{0,1\right\}^{b_\lambda}$ satisfying the first $a_\lambda$ positions is $1-u$ and be seen as an element in $\left\{0,1\right\}^{c_\lambda}$.
\end{lemma}
\begin{proof}
	Similar to above we can assume $u=0$ and $m_u$ can be obtained from $m_0$ for general $u$. When $u=0$ the left side is
	\begin{equation*}
		X^{ind_T(N)+(\sum \lambda_i+1)B_T e_{k-1}+(s-1-a_\lambda+1)B_T e_k+B_T e_{k+2}}.
	\end{equation*}
	Let
	\begin{equation*}
		d=ind_T(N)+(\sum \lambda_i+1)B_T e_{k-1}+(s-a_\lambda)B_T e_k+B_T e_{k+2}.
	\end{equation*}
	Then the left is equal to
	\begin{equation*}
		\sum_{h\in \left\{0,1\right\}^{c_\lambda}}q^{\frac{v_d(h)}{2}}(X')^{d-2c_\lambda e_k+(\sum h_i)(2e_{k-1})+(c_\lambda-(\sum h_i))(e_{k-2}+e_{k+2}) }.
	\end{equation*}
	and the right side is
	\begin{equation*}
		\sum_{h\in \left\{0,1\right\}^{c_\lambda}}q^{\frac{m_0(h)}{2}}(X')^{ind_{T'}(N')+(\sum \lambda_i+1)B_{T'}e_{k-1}+(a_\lambda+\sum h_i+1)B_{T'}e_k+B_{T'}e_{k+2}}.
	\end{equation*}
	Here we see $u'$ as an element of $\left\{0,1\right\}^{c_\lambda}$ with the last $c_\lambda$ coordinate $h$ and identify them. Then we want to show
	\begin{align*}
		d-2c_\lambda e_k+(\sum h_i)(2e_{k-1})+(c_\lambda-(\sum h_i))(e_{k-2}+e_{k+2}) \\
		=ind_{T'}(N')+(\sum \lambda_i+1)B_{T'}e_{k-1}+(a_\lambda+\sum h_i+1)B_{T'}e_k+B_{T'}e_{k+2}.
	\end{align*}
	Compare the entry of both side: the $k$-th entry is followed by
	\begin{equation*}
		c_\lambda-2c_\lambda=s-2(\sum \lambda_i+1)+1.
	\end{equation*}
	For $j\neq k$, the $j$-th entry of left side is 
	\begin{equation*}
		(ind_T(N))_j+(\sum \lambda_i+1)b_{j,k-1}+(s-a_\lambda)b_{jk}+b_{j,k-2}+c_\lambda (\delta_{j,k-2}+\delta_{j,k+2}),
	\end{equation*}
	and the right side is
	\begin{equation*}
		(ind_{T'}(N'))_j+(\sum \lambda_i+1)b'_{j,k-1}+(a_\lambda+1) b'_{jk}+b'_{j,k+2}.
	\end{equation*}
By the following equation, two sides equal
\begin{equation*}
	\begin{aligned}
		(ind_{T'}(N'))_j&=(ind_T(N))_j-s[-b_{jk}]_+ ,\\
		b'_{j,k-1}-b_{j,k-1}&=2[b_{jk}]_+, \\
		b'_{j,k-2}-b_{j,k-2}&=-[-b_{jk}]_+ ,\\
		\delta_{j,k-2}+\delta_{j,k+2}&=[b_{j,k}].
	\end{aligned}
\end{equation*}
Then we get the desired function by taking $m_0=v'_d$.
\end{proof}
Similarly, if we have $\lambda$ with $-d_\lambda:-(s-1)+2\sum \lambda_i<0$. In this way, $a_\lambda=b_\lambda+d_\lambda$.
\begin{lemma}
	For any $u'\in \left\{0,1\right\}^{b_\lambda}$ with $d_\lambda>0$, there exists $m_{u'}:\left\{0,1\right\}^{d_\lambda}\rightarrow \mathbb{Z}$
	\begin{equation*}
		(X')^T(N'^1_{\lambda,u})=\sum_{u}q^{\frac{m_{u'}(u)}{2}}X^T(N^1_{\lambda,u}).
	\end{equation*}
	where $u$ ranges over $\left\{0,1\right\}^{a_\lambda}$ satisfying the first $b_\lambda$ positions is $1-u'$ and be seen as an element in $\left\{0,1\right\}^{d_\lambda}$.
\end{lemma}
For $\lambda$ with $-(s-1)+2\sum \lambda_i=0$. In this way, $a_\lambda=b_\lambda$ and we can get the following equation with similar strategy
\begin{equation*}
	X^T(N_{\lambda,u})=(X')^{T}(N'_{\lambda,1-u}).
\end{equation*}
\par In summary, we have a partition bijection between $CS(N)$ and $CS(N')$. In addition, this bijection always satisfies the polynomial property.
\par Case(8). Assume we have $N:\tau_{k+2}\rightarrow\tau_{k-1}\rightarrow \tau_k \leftarrow \tau_{k-1}\rightarrow \cdots \tau_k\leftarrow \tau_{k-1}\leftarrow \tau_{k+2}$ and $\tau_k$ occurs $s$ times. The other case is nothing but changing the label. Let $N'=\mu_k(N):\tau_{k+2}\rightarrow \tau_k'\rightarrow\tau_{k-1}\leftarrow \tau_k' \rightarrow \tau_{k-1}\leftarrow \cdots \tau_{k-1}\leftarrow \tau_k'\leftarrow \tau_{k+2}$. The only difference between (6) and (8) is the first and last vertices labeled $\tau_{k-2}$. Let $N^0$ be the set of submodules of $N$ without vertex $\tau_{k+2}$, $N^1$ be the set of submodules of $N$ with the first vertex $\tau_{k+2}$, $N^2$ be the set of submodules of $N$ with the last vertex $\tau_{k+2}$, and $N^3$ be the set of submodules of $N$ without vertex $\tau_{k+2}$. Then
\begin{equation*}
	CS(N)=N^0\cup N^1\cup N^2\cup N^3,
\end{equation*}
Dually we have
\begin{equation*}
	CS(N')=N'^0\cup N'^1\cup N'^2\cup N'^3.
\end{equation*}
where $N'^0$ be the set of submodules of $N'$ without vertex $\tau_{k+2}$, $N'^1$ be the set of submodules of $N'$ with the first vertex $\tau_{k+2}$, $N'^2$ be the set of submodules of $N'$ with the last vertex $\tau_{k+2}$, and $N'^3$ be the set of submodules of $N'$ without vertex $\tau_{k+2}$. Then $N^{0}$(or $N'^{0}$) can be identified as $CS(N)$(or $CS(N')$) in case (6). $N^{1}$(or $N'^{1}$) can be identified as $CS(N)$(or $CS(N')$) in the first case of case (7). $N^{2}$(or $N'^{2}$) can be identified as $CS(N)$(or $CS(N')$) in the second case of case (7). So they all have a partition bijection satisfying the polynomial property. The rest is finding a partition bijection between $N^3$ and $N'^3$. For submodule of $N$ with the first and last vertex $\tau_{k+2}$, then the first and last subword $\tau_{k+2}\rightarrow \tau_{k-1}\rightarrow \tau_k$ must be contained. So $N^3$ can be identified with a submodule of $\tau_{k-1}\rightarrow \tau_k \leftarrow \tau_{k-1}\rightarrow \cdots \tau_k\leftarrow \tau_{k-1}$ where $\tau_k$ occurs $s-2$ times in case(6). Then we know
\begin{equation*}
	N^3=\bigcup_{\lambda \in \left\{0,1\right\}^{s-1}}N^3_{\lambda},
\end{equation*}
and for each $\lambda \in \left\{0,1\right\}^{s-1}$, we have the following partition
\begin{equation*}
	N^3_{\lambda}=\bigcup_{u\in \left\{0,1\right\}^{a_\lambda}}N^3_{\lambda,u}.
\end{equation*}
Dually we have
\begin{equation*}
	N'^3=\bigcup_{\lambda \in \left\{0,1\right\}^{s-1}}N'^3_{\lambda},
\end{equation*}
and
\begin{equation*}
	N'^3_{\lambda}=\bigcup_{u'\in \left\{0,1\right\}^{b_\lambda}}N'^3_{\lambda,u'}.
\end{equation*}
Like the above case, we have a partition bijection between $N^3_{\lambda}$ and $N'^3_{\lambda}$ for each $\lambda\in \left\{0,1\right\}^{s-1}$. If $-(s-2)+2(\sum \lambda_i)>0$, we associate $N^3_{\lambda,u}$ with $\varphi_N(N^3_{\lambda,u})$ which is the set of $N'^3_{\lambda,u'}$ where the first $a_\lambda$ coordinate of $u'$ is $1-u$. If $-s+2(\sum\lambda_i)+2\leq 0$, we associate $N'^3_{\lambda,u'}$ with $\phi_N(N'^3_{\lambda,u'})$ which is the set of $N^3_{\lambda,u}$ where the first $b_\lambda$ coordinate of $u$ is $1-u'$. Then the polynomial property is still true. The proof here is totally the same as above. For example, if $c_\lambda=-(s-2)+2(\sum \lambda_i)>0$,
\begin{lemma}
	For any $\lambda$ with $c_\lambda>0$ and $u\in \left\{0,1\right\}^{a_\lambda}$, there exists $m_u:\left\{0,1\right\}^{c_\lambda}\rightarrow \mathbb{Z}$
	\begin{equation*}
		X^T(N^3_{\lambda,u})=\sum_{u'}q^{\frac{m_u(u')}{2}}(X')^T(N'^3_{\lambda,u'}).
	\end{equation*}
	where $u'$ ranges over $\left\{0,1\right\}^{b_\lambda}$ satisfying the first $a_\lambda$ positions is $1-u$ and be seen as an element in $\left\{0,1\right\}^{c_\lambda}$.
\end{lemma}
\begin{proof}
	We can assume $u=0$. Then the left side is
	\begin{equation*}
		X^{ind_T(N)+(\sum \lambda_i+2)B_T e_{k-1}+(s-2-a_\lambda+2)B_T e_k+2B_T e_{k+2}}.
	\end{equation*}
	Let
	\begin{equation*}
		d=ind_T(N)+(\sum \lambda_i+2)B_Te_{k-1}+(s-a_\lambda)B_T e_k+2B_Te_{k+2}.
	\end{equation*}
	Then the left side is 
	\begin{equation*}
		\sum_{h\in \left\{0,1\right\}^{c_\lambda}}q^{\frac{v'_d(h)}{2}}(X')^{d-2c_\lambda e_k+(\sum h_i)(2e_{k-1})+(c_\lambda-(\sum h_i))(e_{k-2}+e_{k+2}) }.
	\end{equation*}
	and the right side is
	\begin{equation*}
		\sum_{h\in \left\{0,1\right\}^{c_\lambda}}q^{\frac{m_u(h)}{2}}(X')^{ind_{T'}(N')+(\sum \lambda_i+2)B_{T'}e_{k-1}+(a_\lambda+\sum h_i+2)B_{T'}e_k+2B_{T'}e_{k+2}}.
	\end{equation*}
	Here we see $u'$ as an element of $\left\{0,1\right\}^{c_\lambda}$ with the last $c_\lambda$ coordinate $h$ and identify them. Then we want to show
	\begin{align*}
		d-2c_\lambda e_k+(\sum h_i)(2e_{k-1})+(c_\lambda-(\sum h_i))(e_{k-2}+e_{k+2}) \\
		=ind_{T'}(N')+(\sum \lambda_i+2)B_{T'}e_{k-1}+(a_\lambda+\sum h_i+2)B_{T'}e_k+2B_{T'}e_{k+2}.
	\end{align*}
	Compare the entry of both side: the $k$-th entry is followed by
	\begin{equation*}
		c_\lambda-2c_\lambda=s-2(\sum \lambda_i+2)+2.
	\end{equation*}
	For $j\neq k$, the $j$-th entry of left side is 
	\begin{equation*}
		(ind_T(N))_j+(\sum \lambda_i+2)b_{j,k-1}+(s-a_\lambda)b_{jk}+2b_{j,k-2}+c_\lambda (\delta_{j,k-2}+\delta_{j,k+2}).
	\end{equation*}
	and the right side is
	\begin{equation*}
		(ind_{T'}(N'))_j+(\sum \lambda_i+2)b'_{j,k-1}+(a_\lambda+2) b'_{jk}+2b'_{j,k+2}.
	\end{equation*}
	By the following equation, two sides equal
	\begin{equation*}
		\begin{aligned}
			(ind_{\Gamma'}(N'))_j&=(ind_\Gamma(N))_j-s[-b_{jk}]_+ ,\\
			b'_{j,k-1}-b_{j,k-1}&=2[b_{jk}]_+, \\
			b'_{j,k-2}-b_{j,k-2}&=-[-b_{jk}]_+ ,\\
			\delta_{j,k-2}+\delta_{j,k+2}&=[b_{j,k}].
		\end{aligned}
	\end{equation*}
	Then we have the desired function by taking $m_0=v'_d$.
\end{proof}
\subsection{Summary} We can divide the string $w$ into different subword such that each subword $N$ does not containing $\tau_{k-2},\tau_{k-1},\tau_k,\tau_{k+1},\tau_{k+2}$ as its vertex or all the vertices are one of them. If $N$ does not containing any of them, then we let $\mu_k(N)=N$. If all the vertices of $N$ are one of them, then it must be one of the above case and we already define $\mu_k(N)$ for each case. 
\begin{proposition}
	Let $N'=\mu_k(N)$. There is a partition bijection $\varphi_N$ between $CS(N)$ and $CS(N')$ satisfying the polynomial property.
\end{proposition}

If we connect all the $\mu_k(N)$ in order then get a new word $\mu_k(w)$ in $Q(\Gamma)$.
\begin{lemma}
	$\mu_k(w)$ is a string in $Q(\Gamma)$. Further more, it corresponding to the same object as $w$ in $\mathcal{C}$.
\end{lemma}
\begin{proof}
	First claim can be proved by check case by case. For each case, the last arrow of the form string and the first arrow of the latter string can not be in a relation. Actually all the case is given by the same subcurve of $r$ in surface of crossings with different triangulation. $w$ is the crossings of $r$ and $\Gamma$ and $\mu_k(w)$ is the crossings of $r$ and $\Gamma'$.
\end{proof}
For a string $w$, let $w'=\mu_k(w)$ and $I$ and $I'$ be the index set of $w$ and $w'$. Using the above rules, we cut $I$ into subsets $I_1,I_2,\cdots,I_p$. Recall we always identify the canonical submodules and subsets of index set. For example, we identify $w$(or $w'$) and $I$(or $I'$). In this case, $I_i$ is not always the submodule of $I$ but we see it as a string module. Then for any submodule $U$ with index set $I_U$ and $1\leq i\leq p$, let $U_i$ be the module with index set $I_U\cap I_i$. Then $U_i$ is a submodule of $I_i$. Similarly $I'$ is cut into $I'_1,I'_2,\cdots,I'_p$ and for any submodule $U'\in CS(w')$, $U'_i=I_{U'}\cap I'_i$.
\begin{proposition}
	There is an injective map $l:CS(w)\rightarrow \prod_{i=1}^{p}CS(I_i)$ given by $U\mapsto (U_i)_{1\leq i\leq p}$. Dually, there is an injective map $l:CS(w')\rightarrow \prod_{i=1}^{p}CS(I'_i)$ given by $U\mapsto (U'_i)_{1\leq i\leq p}$
\end{proposition}
Under this map, we identify $CS(w)$ with a subset of $\prod_{i=1}^{p}CS(I_i)$. So we can write a submodule $U\in CS(w)$ as $(U_i)_{1\leq i\leq p}$ where the following conditions are satisfied
\begin{enumerate}
	\item For any $1\leq j\leq p-1$, if $I_j$ and $I_{j+1}$ are connected by a left arrow $\leftarrow$, then the last vertex of $I_j$ is not in $U_j$ and the first vertex of $I_{j+1}$ is in $U_{j+1}$ can not simultaneously happen.
	\item For any $1\leq j\leq p-1$, if $I_j$ and $I_{j+1}$ are connected by a right arrow $\rightarrow$, then the last vertex of $I_j$ is in $U_j$ and the first vertex of $I_{j+1}$ is not in $U_{j+1}$ can not simultaneously happen.
\end{enumerate}
Similarly we write a submodule $U'\in CS(w')$ as $(U'_i)_{1\leq i\leq p}$ where the following conditions are satisfied
\begin{enumerate}
	\item For any $1\leq j\leq p-1$, if $I_j$ and $I_{j+1}$ are connected by a left arrow $\leftarrow$, then the last vertex of $I_j$ is not in $U'_j$ and the first vertex of $I_{j+1}$ is in $U'_{j+1}$ can not simultaneously happen.
	\item For any $1\leq j\leq p-1$, if $I_j$ and $I_{j+1}$ are connected by a right arrow $\rightarrow$, then the last vertex of $I_j$ is in $U'_j$ and the first vertex of $I_{j+1}$ is not in $U'_{j+1}$ can not simultaneously happen.
\end{enumerate}
We have a partition bijection $\varphi=(\varphi_{I_i})_{1\leq i\leq p}$ from $\prod_{i=1}^{p}CS(I_i)$ to $\prod_{i=1}^{p}CS(I'_i)$ given by the partition bijection $\varphi_{I_i}$ in summary.
\begin{lemma}
	Given $U\in \prod_{i=1}^{p}CS(I_i)$ and $U'\in \prod_{i=1}^{p}CS(I'_i)$. If $U_i\in \varphi_{I_i}(U'_i)$ for any $1\leq i\leq p$, then $U\in CS(w)$ if and only if $U'\in CS(w')$.
\end{lemma}
\begin{proof}
	For any $1\leq j\leq p-1$, if $I_j$ and $I_{j+1}$ are connected by a left arrow $\leftarrow$, then by checking all the case we can see the vertices in the endpoint stay invariant except for the endpoints of $w$. So the conditions are satisfied in $U$ if and only if in $U'$. The other situation is the same.
\end{proof}
\begin{remark}
	From the above lemma, we see $\varphi$ can be restricted to a partition bijection between $CS(w)$ and $CS(w')$. For any $1\leq i\leq p$ and a partition $S_i\in CS(I_i)$, we have a partition of $CS(w)$ given by
	\begin{equation*}
		\left\{U:U_i\in S_i \ for \ any \ 1\leq i\leq p\right\},
	\end{equation*}
	which is denoted by $S_1\times \cdots \times S_p$.
\end{remark}
Then for any $U\in CS(w)$ we have
\begin{equation*}
	X^{ind_{T}(w)+B_{T}dimU}=q^{\frac{\beta}{2}} X^{ind_{T}(I_1)+B_{T}dimU_1}\cdots X^{ind_{T}(I_p)+B_{T}dimU_p}.
\end{equation*}
where
\begin{equation*}
	\beta=-\sum_{1\leq i<j\leq p}\Lambda(U_i,U_j).
\end{equation*}
with $\Lambda(U_i,U_j)=\Lambda(ind_{T}(I_i)+B_{T}dimU_i,ind_{T}(I_j)+B_{T}dimU_j)$.
\par Similarly for any $U'\in CS(w')$ we have
\begin{equation*}
	(X')^{ind_{T'}(w')+B_{T'}dimU'}=q^{\frac{\beta'}{2}} (X')^{ind_{T'}(I'_1)+B_{T'}dimU'_1}\cdots (X')^{ind_{T'}(I'_p)+B_{T'}dimU'_p}.
\end{equation*}
where
\begin{equation*}
	\beta'=-\sum_{1\leq i<j\leq p}\Lambda'(U'_i,U'_j).
\end{equation*}
with $\Lambda'(U'_i,U'_j)=\Lambda'(ind_{T'}(I'_i)+B_{T'}dimU'_i,ind_{T'}(I'_j)+B_{T'}dimU'_j)$.
\begin{definition}For any $U=(U_i)\in CS(w)$ and $U'=(U'_i)\in CS(w')$
	\begin{equation*}
		v(U)=v_{I_1}(U_1)+\cdots+v_{I_p}(U_p)+\sum_{1\leq i<j\leq p}\Lambda(U_i,U_j),
	\end{equation*}
	\begin{equation*}
		v'(U)=v_{I'_1}(U'_1)+\cdots+v_{I'_p}(U'_p)+\sum_{1\leq i<j\leq p}\Lambda'(U'_i,U'_j).
	\end{equation*}
\end{definition}
Then
\begin{equation*}
	\sum_{U\in S_1\times\cdots\times S_p}q^{\frac{v(U)}{2}}X^{T}(U),
\end{equation*}
can be written as
\begin{equation*}
	\sum_{U_1\in S_1}\cdots \sum_{U_p\in S_p}(q^{\frac{v_{I_1}(U_1)}{2}}X^T(U_1))\cdots(q^{\frac{v_{I_p}(U_p)}{2}}X^T(U_p)).
\end{equation*}
Dually, 
\begin{equation*}
	\sum_{U'\in S'_1\times\cdots\times S'_p}q^{\frac{v'(U')}{2}}(X')^{T}(U'),
\end{equation*}
can be written as
\begin{equation*}
	\sum_{U'_1\in S'_1}\cdots \sum_{U'_p\in S'_p}(q^{\frac{v'_{I'_1}(U'_1)}{2}}(X')^T(U'_1))\cdots(q^{\frac{v'_{I'_p}(U'_p)}{2}}(X')^T(U'_p)).
\end{equation*}
\begin{theorem}
	The partition bijection given by $\varphi$ has the polynomial property: For any partition $S$
	\begin{equation*}
		\sum_{U\in S}q^{\frac{v(U)}{2}}X^{ind_{T}(w)+B_{T}dimU}=\sum_{U'\in \varphi(S)}q^{\frac{v'(U')}{2}}(X')^{ind_{T'}(w')+B_{T'}dimU'}.
	\end{equation*}
\end{theorem}
\begin{proof}
	By definition, the partition $S$ must be of form $S_1\times \cdots \times S_p$ where $S_i$ is a partition of $CS(I_i)$ for any $1\leq i\leq p$. Then $\varphi(S)=\varphi_{I_1}(S_1)\times\cdots\times\varphi_{I_p}(S_p)$ and
	\begin{equation*}
		\sum_{U_i\in S_i}q^{\frac{v_{I_i}(U_i)}{2}}X^T(U_i)=\sum_{U'_i\in \varphi_{I_i}(S_i)}q^{\frac{v_{I'_i}(U'_i)}{2}}(X')^T(U'_i).
	\end{equation*}
	So
	\begin{align*}
		\sum_{U\in S}q^{\frac{v(U)}{2}}X^{ind_{T}(w)+B_{T}dimU}=\sum_{U\in S_1\times\cdots\times S_p}q^{\frac{v(U)}{2}}X^{T}(U).
	\end{align*}
	is equal to
	\begin{equation*}
		\sum_{U_1\in S_1}\cdots \sum_{U_p\in S_p}(q^{\frac{v_{I_1}(U_1)}{2}}X^T(U_1))\cdots(q^{\frac{v_{I_p}(U_p)}{2}}X^T(U_p)).
	\end{equation*}
	By the construction in summary, we know the above equals
	\begin{equation*}
		\sum_{U'_1\in \varphi_{I_1}(S_1)}\cdots \sum_{U'_p\in \varphi_{I_p}(S_p)}(q^{\frac{v'_{I'_1}(U'_1)}{2}}(X')^T(U'_1))\cdots(q^{\frac{v'_{I'_p}(U'_p)}{2}}(X')^T(U'_p)).
	\end{equation*}
	By the definition of $v'$, this equals
	\begin{equation*}
		\sum_{U'\in \varphi_{I_1}(S_1)\times\cdots\times \varphi_{I_p}(S_p)}q^{\frac{v'(U')}{2}}(X')^T(U').
	\end{equation*}
	which is the desired form.
\end{proof}
Then we can get the following corollary by the definition of partition bijection
\begin{theorem}
	There exists functions $v$ and $v'$ valued on $\mathbb{Z}$ satisfying
	\begin{equation*}
		\sum_{U\in CS(M(w))}q^{\frac{v(U)}{2}X^{ind_T(w)+B_T dim(U)}}=\sum_{U'\in CS(M(w'))}q^{\frac{v'(U')}{2}(X')^{ind_{T'}(w')+B_{T'} dim(U')}}.
	\end{equation*}
\end{theorem}

\subsection{Example}
Consider an annulus with three marked points, and choose an appropriate triangulation $\Gamma$ such that $Q(\Gamma)$ is \begin{tikzcd}
	1 \arrow[r] \arrow[rr, bend left] & 2 \arrow[r] & 3
\end{tikzcd}, see \cite{Fomin2006ClusterAA} for additional details. Take the string $w:1\rightarrow 2$ and consider the mutation at vertex 1. Then the new quiver is \begin{tikzcd}
1' & 2 \arrow[l] \arrow[r] & 3 \arrow[ll, bend left]
\end{tikzcd}. And the corresponding string $w'$ is the trivial string with a single vertex labeled 2. This is the case $(4)$ in Section 5.2 and $B(\Gamma)=\begin{pmatrix}
0 & -1 & -1 \\
1 & 0 & -1 \\
1 & 1 & 0
\end{pmatrix}$. 
\par The string module $M(w)$ is the injective module $I(2)$. Let $T=T_1\oplus T_2\oplus T_3$ be the cluster tilting object. By the identity $M(w)=Ext^1_{\mathcal{C}}(T,T_2[1])$, we know that the corresponding string object in $\mathcal{C}$ is $T_2[1]$, and its index is $-e_2$. The string module $M(w)$ has three canonical submodules $N_1, N_2, N_3$, with index sets $\emptyset, \left\{2\right\}, \left\{1,2\right\}$, respectively.
\begin{equation*}
	X^T(N_1)=X^{ind_T(M(w))+B(\Gamma)0}=X^{-e_2}.
\end{equation*}
\begin{equation*}
	X^T(N_2)=X^{ind_T(M(w))+B(\Gamma)e_2}=X^{-e_2-e_1+e_3}.
\end{equation*}
\begin{equation*}
	X^T(N_3)=X^{ind_T(M(w))+B(\Gamma)(e_1+e_2)}=X^{-e_1+2e_3}.
\end{equation*}
Therefore, the expansion formula is
\begin{equation*}
	X^T_{M(w)}=X^T(N_1)+X^T(N_2)+X^T(N_3)=X^{-e_2}+X^{-e_2-e_1+e_3}+X^{-e_1+2e_3}.
\end{equation*}
Similarly, let $\Gamma'$ be the new triangulation and $B(\Gamma')=\begin{pmatrix}
	0 & 1 & 1 \\
	-1 & 0 & -1 \\
	-1 & 1 & 0
\end{pmatrix}$. Let $w'$ be the trivial string with a single vertex labeled 2. Then the string module $M(w')$ is the injective module associated with the vertex 2. For the same reason, its index is $-e_2$. The string module $M(w')$ has two canonical submodules $N'_1, N'_2$,  whose index sets are $\emptyset, \left\{1\right\}$, respectively.
\begin{equation*}
	(X')^T(N'_1)=(X')^{ind_{T'}(M(w'))+B(\Gamma')0}=(X')^{-e_2}.
\end{equation*}
\begin{equation*}
	(X')^T(N'_2)=(X')^{ind_{T'}(M(w'))+B(\Gamma')e_1}=(X')^{-e_2+e_1+e_3}.
\end{equation*}
Hence, the expansion formula is
\begin{equation*}
	X^{T'}_{M(w')}=	(X')^T(N'_1)+(X')^T(N'_2)=(X')^{-e_2}+(X')^{-e_2+e_1+e_3}.
\end{equation*}
By the mutation rules given in Definition 3.1, we have $X^T_{M(w)}=X^{T'}_{M(w')}$.

\section*{Acknowledgement}{Fan Xu was supported by Natural Science Foundation of China
	[Grant No. 12031007].}

\nocite{*}
\bibliographystyle{plain}
\bibliography{quantum_cluster_variables}
\end{document}